\documentclass[11pt]{article}
\makeatletter

\usepackage{amssymb,amsmath,amsthm,latexsym}
\usepackage[mathscr]{eucal} \usepackage{mathrsfs}
\usepackage{color}
\usepackage{cases}
\usepackage{hyperref}
\hypersetup{
  colorlinks   = true, 
  urlcolor     = blue, 
  linkcolor    = blue, 
  citecolor   = red 
}
\usepackage{comment}
\let\wfs@comment@comment\comment
\let\comment\@undefined

\usepackage{changes}
\let\wfs@changes@comment\comment
\let\comment\@undefined

\newcommand\comment{%
    \ifthenelse{\equal{\@currenvir}{comment}}
    {\wfs@comment@comment}
    {\wfs@changes@comment}%
}
\definechangesauthor[name=Daniele, color=red]{DAN}
\definechangesauthor[name=Giuseppe, color=green]{GIUS}
\title{Exceptional scattered sequences}
\usepackage{authblk}
\author[1]{Daniele Bartoli}
\affil[1]{Department of Mathematics and Informatics, University of Perugia, Perugia,  Italy, {\small \texttt{daniele.bartoli@unipg.it}}\vspace*{.3cm}}
\author[2]{Giuseppe Marino}
\affil[2]{Department of Mathematics and Applications ``R. Caccioppoli'', University of Naples ``Federico II'', Napoli, Italy, {\small\texttt{giuseppe.marino@unina.it}}\vspace*{.3cm}}
\author[3]{Alessandro Neri}
\affil[3]{Max-Planck-Institute for Mathematics in the Sciences, Leipzig, Germany, {\small\texttt{alessandro.neri@mis.mpg.de}}\vspace*{.3cm}}
\author[4]{Lara Vicino}
\affil[4]{Department of Applied Mathematics and Computer Science, Technical University of Denmark, Kgs. Lyngby, Denmark, 
{\small\texttt{lavi@dtu.dk}}}
\date{}

\DeclareMathOperator{\dd}{d}
\DeclareMathOperator{\rk}{rk}
\DeclareMathOperator{\Hom}{Hom}
\DeclareMathOperator{\GL}{GL}
\DeclareMathOperator{\End}{End}
\DeclareMathOperator{\ev}{ev}
\DeclareMathOperator{\Gal}{Gal}

\begin{document}
\maketitle

\theoremstyle{definition}
\newtheorem{theorem}{Theorem}[section]
\newtheorem{lemma}[theorem]{Lemma}
\newtheorem{conj}[theorem]{Conjecture}
\newtheorem{remark}[theorem]{Remark}
\newtheorem{corollary}[theorem]{Corollary}
\newtheorem{prop}[theorem]{Proposition}
\newtheorem{definition}[theorem]{Definition}
\newtheorem{result}[theorem]{Result}
\newtheorem{property}[theorem]{Property}
\newtheorem{problem}[theorem]{Open Problem}

\newtheorem{question}[theorem]{Question}

\makeatother
\newcommand{\Prf}{\noindent{\bf Proof}.\quad }
\renewcommand{\labelenumi}{(\alph{enumi})}


\def\B{\mathbf B}
\def\C{\mathbf C}
\def\Q{\mathbf Q}
\def\W{\mathbf W}
\def\a{\mathbf a}
\def\b{\mathbf b}
\def\c{\mathbf c}
\def\d{\mathbf d}
\def\e{\mathbf e}
\def\l{\mathbf l}
\def\v{\mathbf v}
\def\w{\mathbf w}
\def\x{\mathbf x}
\def\y{\mathbf y}
\def\z{\mathbf z}
\def\t{\mathbf t}
\def\cD{\mathcal D}
\def\cC{\mathcal C}
\def\cH{\mathcal H}
\def\cM{{\mathcal M}}
\def\cK{\mathcal K}
\def\cQ{\mathcal Q}
\def\cU{\mathcal U}
\def\cS{\mathcal S}
\def\cT{\mathcal T}
\def\cR{\mathcal R}
\def\cN{\mathcal N}
\def\cA{\mathcal A}
\def\cF{\mathcal F}
\def\cL{\mathcal L}
\def\cB{\mathcal B}
\def\cP{\mathcal P}
\def\cG{\mathcal G}
\def\cGD{\mathcal GD}
\def\mC{\mathcal C}
\def\mU{\mathcal U}

\def\PG{{\rm PG}}
\def\GF{{\rm GF}}
\def\Tr{{\rm Tr}}

\def\Pg{PG(5,q)}
\def\pg{PG(3,q^2)}
\def\ppg{PG(3,q)}
\def\HH{{\cal H}(2,q^2)}
\def\F{\mathbb F}
\def\Ft{\mathbb F_{q^t}}
\def\P{\mathbb P}
\def\Z{\mathbb Z}
\def\V{\mathbb V}
\def\bS{\mathbb S}
\def\G{\mathbb G}
\def\E{\mathbb E}
\def\N{\mathbb N}
\def\K{\mathbb K}
\def\D{\mathbb D}
\def\ps@headings{
 \def\@oddhead{\footnotesize\rm\hfill\runningheadodd\hfill\thepage}
 \def\@evenhead{\footnotesize\rm\thepage\hfill\runningheadeven\hfill}
 \def\@oddfoot{}
 \def\@evenfoot{\@oddfoot}
}
\def\cub{\mathscr C}
\def\cO{\mathcal O}
\def\cur{\mathscr L}
\def\Fqm{{\mathbb F}_{q^m}}
\def\Fq3{{\mathbb F}_{q^3}}
\def\fq{{\mathbb F}_{q}}
\def\Fm{{\mathbb F}_{q^m}}
\def\Fn{{\mathbb F}_{q^n}}

\newcommand{\Fmk}{[n,k]_{q^m/q}}
\newcommand{\Fmkd}{[n,k,d]_{q^m/q}}
\newcommand{\Fmkdd}{[n,k,(d_1,\ldots,d_k)]_{q^m/q}}
\newcommand{\ale}[1]{{\color{blue}[$\star \star$ {\sf Alessandro: #1}]}}
\newcommand{\dan}[1]{{\color{violet}[$\star \star$ {\sf Daniele: #1}]}}
\newcommand{\giu}[1]{{\color{red}[$\star \star$ {\sf Giuseppe: #1}]}}
\newcommand{\lara}[1]{{\color{cyan}[$\star \star$ {\sf Lara: #1}]}}
\newcommand{\Fms}[3]{[#1,#2]_{q^{#3}/q}}

\begin{abstract}
The concept of scattered polynomials is generalized to those of exceptional scattered sequences which are shown to be the  natural algebraic counterpart of $\mathbb{F}_{q^n}$-linear MRD codes. The first infinite family in the first nontrivial case is also provided and equivalence issues are considered. As a byproduct, a new infinite family of MRD codes is obtained.
\end{abstract}

\bigskip


\section{Introduction}

Rank-metric codes  were introduced already in the late 70's by Delsarte \cite{delsarte1978bilinear} and then rediscovered by Gabidulin a few years later \cite{gabidulin1985theory}. They attracted many researchers in the last decade, due to their applications in network coding \cite{MR2450762} and cryptography \cite{gabidulin1991ideals,MR3678916}.
Such codes are sets of matrices over a finite field $\fq$ endowed with the rank distance, that is, the distance  between two elements is defined as the rank of their difference. Among them, of particular interest is the family of rank-metric codes whose parameters are optimal, that is, for the given minimum rank, they have the maximum possible cardinality. Such codes are called \emph{maximum rank distance (MRD) codes} and constructing new families is an important and active research task.
From a different perspective, rank-metric codes can also  be seen as sets of (restrictions of) $\mathbb{F}_q$-linear homomorphisms from $(\mathbb{F}_{q^n})^m$ to $\mathbb{F}_{q^n}$ equipped with the rank distance; see Sections \ref{Sec:2.2} and \ref{Sec:2.3}.
With this second point of view, it is evident that multivariate linearized polynomials can be seen as the natural algebraic counterpart of rank-metric codes. In the case of univariate linearized polynomials such a connection was exploited in \cite{sheekey2016new} by Sheekey, where the notion of scattered polynomials was introduced; see also \cite{bartoli2018exceptional}. Let $f \in \mathcal{L}_{n,q}[X]$ be a $q$-linearized polynomial and let $t$ be a nonnegative integer with $t\leq n-1$. Then, $f$ is said to be \emph{scattered of index $t$} if for every $x,y \in \mathbb{F}_{q^n}^*$
\[ \frac{f(x)}{x^{q^t}}=\frac{f(y)}{y^{q^t}}\,\, \Longleftrightarrow\,\, \frac{y}x\in \fq, \]
or equivalently
\begin{equation*} \dim_{\fq}(\ker(f(x)-\alpha x^{q^t}))\leq 1, \,\,\,\text{for every}\,\,\, \alpha \in \mathbb{F}_{q^n}. \end{equation*}
In a more geometrical setting, a scattered polynomial is connected with a scattered subspace of the projective line; see \cite{blokhuis2000scattered}. From a coding theory point of view,  $f$ is scattered of index $t$ if and only if $\cC_{f,t}=\langle x^{q^t}, f(x) \rangle_{\mathbb{F}_{q^n}}$ is a \emph{maximum rank distance (MRD) code} with $\dim_{\mathbb{F}_{q^n}}(\cC_{f,t})=2$.
The polynomial $f$ is said to be \emph{exceptional scattered} of index $t$ if it is scattered of index $t$ as a polynomial in $\mathcal{L}_{\ell n,q}[X]$, for infinitely many $\ell$; see \cite{bartoli2018exceptional}. The classification of exceptional scattered polynomials is still not complete, although it gained the attention of several researchers \cite{Bartoli:2020aa4,bartoli2018exceptional,MR4190573,MR4163074,bartoli2022towards}.

While many families of scattered polynomials have been constructed in recent years \cite{sheekey2016new,lunardon2018generalized,lunardon2000blocking,zanella2019condition,MR4173668,longobardi2021linear,longobardi2021large,NPZ,zanella2020vertex,csajbok2018new,csajbok2018new2,marino2020mrd,blokhuis2000scattered}, only two families of exceptional ones are known:
\begin{itemize}
    \item[(Ps)] $f(x)=x^{q^s}$ of index $0$, with $\gcd(s,n)=1$ (polynomials of so-called pseudoregulus type);
    \item[(LP)] $f(x)=x+\delta x^{q^{2s}}$ of index $s$, with $\gcd(s,n)=1$ and $\mathrm{N}_{q^n/q}(\delta)\ne1$ (so-called LP polynomials).
\end{itemize}

The generalization of the notion of exceptional scattered polynomials -- together with  their connection with $\mathbb{F}_{q^n}$-linear MRD codes of $\Fn$-dimension $2$ -- yielded the introduction of the concept of 
 $\mathbb{F}_{q^n}$-linear MRD codes of \emph{exceptional type}; see \cite{bartoli2021linear}.
An $\mathbb{F}_{q^n}$-linear MRD code $\cC\subseteq\mathcal{L}_{n,q}[X]$ is an \emph{exceptional MRD code} if the rank metric code
\[ \cC_\ell=\langle\mathcal{C}\rangle_{\mathbb{F}_{q^{\ell n}}}\subseteq \mathcal{L}_{\ell n,q}[X] \]
is an MRD code for infinitely many $\ell$.
Only two families of exceptional $\mathbb{F}_{q^n}$-linear MRD codes are known:
\begin{itemize}
    \item[(G)] $\mathcal{G}_{k,s}=\langle x,x^{q^s},\ldots,x^{q^{s(k-1)}}\rangle_{\mathbb{F}_{q^n}}$, with $\gcd(s,n)=1$; see \cite{delsarte1978bilinear,gabidulin1985theory,kshevetskiy2005new};
    \item[(T)] $\mathcal{H}_{k,s}(\delta)=\langle x^{q^s},\ldots,x^{q^{s(k-1)}},x+\delta x^{q^{sk}}\rangle_{\mathbb{F}_{q^n}}$, with $\gcd(s,n)=1$ and $\mathrm{N}_{q^n/q}(\delta)\neq (-1)^{nk}$; see  \cite{sheekey2016new,lunardon2018generalized}. 
\end{itemize}
The first family is known as \emph{generalized Gabidulin codes} and the second one as \emph{generalized twisted Gabidulin codes}, whereas in \cite{bartoli2018exceptional} it has been shown that the only exceptional $\mathbb{F}_{q^n}$-linear MRD codes spanned by monomials are the codes (G), in connection with so-called \emph{Moore exponent sets}. Non-existence results on exceptional MRD codes were provided in \cite[Main Theorem]{bartoli2021linear}.

\medskip

In this paper we introduce the new notions of $h$-scattered sequences and exceptional $h$-scattered sequences  which constitute the right environment for exceptional MRD codes. These $h$-scattered sequences are sequences of multivariate linearized polynomials $\mathcal F=(f_1,\ldots, f_s)\in \cL_{n,q}[X_1,\ldots,X_m]$, such that there exists 
$\mathcal I=(i_1,\ldots, i_m) \in \N^m$ so that the space 
$$ U_{\mathcal I,\mathcal F}:=\{ (x_1^{i_1}, \ldots,x_m^{i_m},f_1(x_1,\ldots,x_m), \ldots,f_s(x_1,\ldots,x_m) ) \, : \, x_1,\ldots,x_m \in \Fn\}$$
is $h$-scattered; see Definitions \ref{Def:ScatteredSequence} and \ref{Def:ExceptionalScatteredSequence}. 
We then focus on the concept of indecomposability of $h$-scattered sequences, which ensures that they cannot be obtained as direct sums of smaller $h$-scattered sequences, and study how this property is preserved under classical and Delsarte dualities.  
Finally we introduce  the sequences of multivariate linearized polynomials 
$$  (X^{q^I}+\alpha X^{q^J}, X^{q^J}+\beta Y^{q^I}+\gamma Y^{q^J}),$$
for $I,J\in\{1,\ldots,n-1\}$ and $\alpha, \beta,\gamma\in \Fn^*$, and study their associated subspaces 
$$    U_{\alpha, \beta, \gamma}^{I,J,n}:=\left\{\left(x,y,x^{q^I}+\alpha y^{q^J},x^{q^J}+\beta y^{q^I} + \gamma y^{q^J}\right) \ : \ x,y\in \F_{q^n}\right\}\subseteq \Fn^4.$$
We show in Theorem \ref{thm:scattered:1} that if a certain polynomial -- which depends on $I,J,\alpha,\beta,\gamma$ -- has no roots in $\Fn$, then this sequence is $1$-scattered.  This condition is also necessary when restricting to $I,J\leq n/4$, as we observe in Theorem \ref{thm:scattered:2}. 
We then estimate the maximum $\fq$-dimension intersection of $U_{\alpha, \beta, \gamma}^{I,J,n}$ with the $2$-dimensional $\Fn$-subspaces of $V(4,q^n)$. As a byproduct, this gives an estimate on all the generalized rank weights of the code $C_{\alpha, \beta, \gamma}^{I,J,n}$ associated with $U_{\alpha, \beta, \gamma}^{I,J,n}$. In particular, we observe that whenever $\max\{I,J\}\leq (n-1)/2$, our construction automatically produces new MRD codes which are inequivalent from the known constructions and whose generalized rank weights are larger than the ones of the known constructions. We finally investigate equivalence and dualities of the $\fq$-subspaces $U_{\alpha, \beta, \gamma}^{I,J,n}$.

The paper is structured as follows. Section \ref{sec:preliminaries} contains the preliminary notions needed throughout the paper. In particular, we  describe algebraic curves over finite fields, multivariate linearized polynomials, rank-metric codes, and the concepts of evasive and scattered subspaces. In Section \ref{sec:scattered_sequences} we introduce $h$-scattered sequences of multivariate linearized polynomials and the concepts of exceptionality and indecomposability. Section \ref{sec:construction} is devoted to the main general family of a scattered sequence of bivariate linearized polynomials, and the study of its properties. Finally, we draw our conclusions in  Section \ref{sec:conclusions}, describing some open problems.

\section{Definitions and preliminary results}\label{sec:preliminaries}

\subsection{Algebraic curves over a finite field}\label{sec:curves}

In this subsection, we collect some preliminary definitions and results on algebraic curves over a finite field.
Let $q=p^h$, where $p$ is a prime and $h>0$ an integer, and denote by $\F_q$ the finite field with $q$ elements. We denote by $\overline{\F}_q$ the algebraic closure of $\F_q$ and by $\F_q[X,Y]$ the ring of polynomials in the variables $X$ and $Y$ with coefficients in $\F_q$. Finally, let $\mathbb{P}^r(\mathbb{F}_q)$ and $\mathbb{A}^r(\mathbb{F}_q)$ denote, respectively, the $r$-dimensional projective and affine space over $\mathbb{F}_q$.
A curve is a variety of dimension $1$ and plane curves are defined by bivariate polynomials $f(X,Y)\in \mathbb{F}_q[X,Y]$.

Let $\mathcal{X}$ be an irreducible algebraic curve in $\mathbb{P}^n(\F_q)$ and
let $\mathcal{X}(\F_q)$ (resp. $\mathcal{X}(\overline{\F}_q)$) denote the set of all the places of $\mathcal{X}$ defined over $\mathbb{F}_q$ (resp. $\overline{\F}_q$). For a more comprehensive introduction to algebraic varieties and curves we refer the interested reader to \cite{HKT,Hartshorne,Stichtenoth}.

We recall now the following result, defining a \emph{Kummer cover} of a plane curve.
\begin{theorem}[\textnormal{\cite[Corollary 3.7.4]{Stichtenoth}}]
\label{thm:kummer}
Let $\mathcal{X}: F(X,Y)=0$ be an absolutely irreducible plane curve defined over $\F_q$ and let $\overline{\mathcal{X}}$ be its projective closure in $\mathbb{P}^2(\F_q)$. Let $m$ be a positive integer such that $\mathrm{gcd}(m,p)=1$ and $f(X,Y)\in \F_q(\overline{\mathcal{X}})$ be such that there exists a place $Q\in \overline{\mathcal{X}}(\overline{\F}_q)$ with $\mathrm{gcd}(v_Q(f),m)=1$, where $v_Q(f)$ denotes the valuation at $Q$ of the rational function $f$. 
Let $\mathcal{X}'$ be the space curve defined by the following equations
\begin{equation*}
    \mathcal{X}': \begin{cases}
    F(X,Y) = 0\\
    Z^m = f(X,Y)
    \end{cases}
\end{equation*}
and let $\overline{\mathcal{X}'}$ be its projective closure in $\mathbb{P}^3(\F_q)$.
Then $\overline{\mathcal{X}'}$ is an absolutely irreducible curve defined over $\F_q$ and it is called a \emph{Kummer cover} of $\overline{\mathcal{X}}$. Correspondingly, $\mathcal{X}'$ is called a Kummer cover of $\mathcal{X}$.
\end{theorem} 

Note that  Theorem \ref{thm:kummer} applies in particular if $\mathcal{X}$ is a line, in which case $\mathcal{X}'$ is a plane curve.
As it is shown in \cite[Corollary 3.7.4]{Stichtenoth}, if the genus of $\mathcal{X}$ is given, then it is possible to easily compute the genus of a Kummer cover $\mathcal{X}'$.

Finally, we conclude this subsection stating the well-known Hasse-Weil bound for the number of $\F_q$-rational places of a curve defined over $\F_q$.

\begin{theorem}[Hasse-Weil]
\label{thm:hasse-weil}
Let $\mathcal{X}$ be an absolutely irreducible algebraic curve in $\mathbb{P}^n(\F_q)$ of genus $g$. Then the set $\mathcal{X}(\F_q)$ of its $\mathbb{F}_q$-rational places satisfies
\begin{equation}\label{EQ:HW2}
    q + 1 - 2g\sqrt{q} \leq |\mathcal{X}(\F_q)| \leq q + 1 + 2g\sqrt{q}.
\end{equation}
\end{theorem}

If the curve $\mathcal{X}$ is singular, there is some ambiguity in defining what an $\mathbb{F}_q$-rational point of $\mathcal{X}$ actually is. For this reason often the function field version is also used; see \cite{Stichtenoth}. The difference between the number of $\mathbb{F}_q$-rational points of a non-singular model $\mathcal{X}^{\prime}\subset \mathbb{P}^r(\mathbb{F}_q)$, for some integer $r$, of $\mathcal{X}$ and the number of  ``true" $\mathbb{F}_q$-rational  points $(x_0:y_0:t_0)\in \mathbb{P}^2(\mathbb{F}_q)$ of $\mathcal{X}$ is at most $(d-1)(d-2)/2-g$; see \cite[Lemma 9.55]{HKT}. We refer the interested readers to \cite[Section 9.6]{HKT}, where other relations are investigated. Thus, since we will be interested in solutions of particular equations (which correspond to centers of $\mathbb{F}_q$-rational places) we can roughly say that for an absolutely irreducible curve  defined over $\mathbb{F}_q$ the condition  $q + 1 - 2g\sqrt{q}>0$ still yields the existence of at least one $\mathbb{F}_q$-rational  point $(x_0:y_0:t_0)\in \mathbb{P}^2(\mathbb{F}_q)$ (seen as the center of at least one  $\mathbb{F}_q$-rational place).

\subsection{The space of multivariate linearized polynomials}\label{Sec:2.2}
Linearized polynomials over finite fields are important objects with a rich literature, for both a theoretical and an applied point of view. 
Formally, one defines the set of $q$-polynomials over a finite field $\Fn$ as

$$ L_{n,q}[X]:= \Big\{ \sum_{j=0}^t a_jX^{q^j} \,:\, a_j \in \Fn\Big\}.$$

This set can be naturally considered as a ring $(\mathcal L, +, \circ)$, endowed with standard polynomial addition ($+$) and  polynomial map composition ($\circ$). The importance of this ring is due to the fact that the polynomial evaluation map provides an $\fq$-algebra isomorphism
\begin{equation}\label{eq:poly_isomorphism}\mathcal L_{n,q}[X] := L_{n,q}[X]/(X^{q^n}-X)\cong \End_{\fq}(\Fn).
\end{equation}

In this section we study a natural extension of the ring of linearized polynomials to the multivariate setting. Define the set of formal multivariate linearized polynomials on $m$ variables as the $\Fn$-vector space over the (infinite) basis
$$ \{X_i^{q^j} \, :\, 1 \leq i \leq m, j \in \mathbb N\}.$$
In order to mimic the action of the generator of $\Gal(\Fn/\fq)$, we then reduce this vector space modulo the relations 
$$ \{X_i^{q^n}-X_i =0\, : \, 1 \le i \le m \}. $$
In this way, we obtain the following $\Fn$-vector space.
Let $\underline{X}:=(X_1,\ldots,X_m)$ be a vector of indeterminates and let 
\begin{align*} \cL_{n,q}[\underline{X}]:= & \Big\{ \sum_{i=1}^m\sum_{j=0}^{n-1} f_{i,j}X_i^{q^j} \,:\, f_{i,j}\in\Fn \Big\}
= \left\langle \big\{ X_i^{q^j} \, :\, 1\le i \le m, 0\le j \le n-1\big\} \right\rangle_{\Fn}. 
\end{align*}

The following result gives a linearized polynomial representation of spaces of rectangular matrices.

\begin{prop}\label{prop:isomorphism}
The polynomial evaluation map given by
$$ 
\begin{array}{rcl} 
\cL_{n,q}[\underline{X}]&\longrightarrow  &\Hom_{\fq}((\Fn)^m,\Fn) \\
 f & \longmapsto & (v \longmapsto f(v))
\end{array}$$
is an isomorphism of $\fq$-vector spaces.
\end{prop}

\begin{proof}
By definition, we have that
$$\cL_{n,q}[\underline{X}]\cong \bigoplus_{i=1}^m\cL_{n,q}[X_i],$$
 as $\fq$-vector space. Combining it with \eqref{eq:poly_isomorphism}, we obtain
\begin{equation*}\cL_{n,q}[\underline{X}]\cong \bigoplus_{i=1}^m \Hom_{\fq}(\Fn,\Fn)\cong \Hom_{\fq}((\Fn)^m,\Fn) \cong \fq^{n \times nm}.
\end{equation*}
\end{proof}

Due to Proposition \ref{prop:isomorphism}, we can define the \textbf{rank} of a multivariate linearized polynomial in $\cL_{n,q}[\underline{X}]$ as the $\fq$-{rank} of the associated $\fq$-linear homomorphism from $(\Fn)^m$ to $\Fn$. Like for the univariate case, it is immediate to see that rank-one linearized multivariate polynomials can all be expressed in terms of the field trace. In the sequel, let $\Tr_{{q^n}/q}$ denote the \textbf{trace function} of $\Fn$ over $\fq$.

\begin{lemma}\label{lem:rankone}
 $$\{f \in \cL_{n,q}[\underline{X}] \,:\, \rk(f)=1\}=\{\alpha \Tr_{q^n/q}(v \underline{X}^\top ) \,: \alpha \in \Fn^*, v \in (\Fn)^m\setminus \{0\}\}. $$
\end{lemma}

Consider the $\Fn$-bilinear form on $\cL_{n,q}[\underline{X}]$, given by $f\star g:=\sum_{i,j}f_{i,j}g_{i,j}$, where
$$ f=\sum_{i=1}^m\sum_{j=0}^{n-1}f_{i,j}X_i^{q^j}, \qquad g=\sum_{i=1}^m\sum_{j=0}^{n-1}g_{i,j}X_i^{q^j}.$$

\begin{lemma}
Let $f \in \cL_{n,q}[\underline{X}]$, and let $\alpha \in \Fn^*$ and $v \in (\Fn)^m$. Then 
$$ f \star (\alpha \Tr_{q^n/q}(v \underline{X}^\top ))=\alpha f(v).$$
\end{lemma}

\begin{proof}
Since $ f \star (\alpha \Tr_{q^n/q}(v \underline{X}^\top ))=\alpha  f \star ( \Tr_{q^n/q}(v \underline{X}^\top ))$, it is enough to prove it for $\alpha=1$. This is a straightforward computation, since, writing $v=(v_1,\ldots, v_m)$ and  $f=\sum_{i,j}f_{i,j}X_i^{q^j}$, we have 
$$  f \star ( \Tr_{q^n/q}(v \underline{X}^\top ))=\sum_{i=1}^m\sum_{j=0}^{n-1}f_{i,j}v_i^{q^j}=f(v).$$
\end{proof}

\subsection{Rank-Metric Codes}\label{Sec:2.3}
Since we have seen in Proposition \ref{prop:isomorphism}  that the space of multivariate linearized polynomials over $\Fn$ is isomorphic to the space of $n\times nm$ matrices over $\fq$, we can actually study rank-metric codes in $\cL_{n,q}[\underline{X}]$. Here, we define the \textbf{rank distance}  to be the distance $\dd_{\rk}$ induced by the rank:
$$ \dd_{\rk}(f,g):=\rk(f-g).$$

\begin{definition}
An $\Fn$-linear rank-metric code $\cC$ is an $\Fn$-subspace of $\cL_{n,q}[\underline{X}]$, endowed with the rank metric. The \textbf{dimension} of $\cC$ is $k=\dim_{\Fn}(\cC)$ and its \textbf{minimum rank distance} is the integer
$$d=\dd_{\rk}(\cC):=\min\{\rk(f) \,:\, f \in \cC\setminus\{0\}\}.$$
\end{definition}

The parameters of an $\Fn$-linear rank-metric code in $\cL_{n,q}[\underline{X}]$ must satisfy the following  inequality, known as the Singleton-like bound, which was shown by Delsarte in \cite{delsarte1978bilinear}:
\begin{equation}\label{eq:sing_bound}
     k \leq m(n-d+1).
\end{equation}
Codes meeting \eqref{eq:sing_bound} with equality are called \textbf{maximum rank distance (MRD) codes}.

Let $\cC\subseteq \cL_{n,q}[\underline{X}]$ be an $\Fn$-linear code. The \textbf{dual code} is $$\cC^\perp=\{ f \in \cL_{n,q}[\underline{X}] \, :\, f\star g = 0 \mbox{ for all } g \in \cC\}.$$

Apart from a classical representation as matrices over $\fq$, rank-metric codes are also usually represented as spaces of vectors over the extension field $\Fn$, especially when they have an inherited $\Fn$-linearity.  The way to connect our codes in $\cL_{n,q}[\underline{X}]$ with those in $V(nm,q^n)$ is briefly described as follows. Let us fix an $\fq$-basis $(\beta_1,\ldots,\beta_n)$ of $\Fn$, and take the $\fq$-basis of $(\Fn)^m$ given by 
\begin{equation}\label{eq:basis}\cB:=(\beta_je_i)_{\substack{1\leq i \leq m,\\ 1 \leq j \leq n}},\end{equation}
where $\{e_1,\ldots,e_m\}$ is the canonical $\mathbb{F}_{q^n}$-basis of $(\mathbb{F}_{q^n})^m$.

Define the map
$$\begin{array}{rccl}\ev_{\cB}:& \cL_{n,q}[\underline{X}] & \longrightarrow & (\Fn)^{nm} \\
&f & \longmapsto & (f(\beta_je_i))_{\substack{1\leq i \leq m,\\ 1 \leq j \leq n}}.\end{array}$$
From Proposition \ref{prop:isomorphism}, we immediately deduce the following.
\begin{corollary}
The map $\ev_{\cB}$ is an $\Fn$-linear isomorphism.
\end{corollary}
Let $\cG=(g_1,\ldots, g_k)$ be an $\Fn$-basis of $\cC$. We define the $\fq$-space
$$ U_{\cG}:=\{(g_1(x_1,\ldots,x_m), \ldots,g_k(x_1,\ldots,x_m) ) \, : \, x_1,\ldots,x_m \in \Fn \}\subseteq V(k,q^n). $$

\begin{definition}
 Let $\cC$ be a $k$-dimensional $\Fn$-linear code, and let $\cG$ be a basis of $\cC$.  The  \textbf{effective length} of  $\cC$ is $\ell(\cC):=\dim_{\fq}(U_{\cG})$. 
The code $\cC$ is \textbf{nondegenerate} if  $\ell(\cC)=nm$. 
\end{definition}

\begin{remark}
 The effective length of a code is well-defined. Indeed, while the $\fq$-space $U_{\cG}$ depends on the choice of the $\Fn$-basis $\cG$ of $\cC$, its $\fq$-dimension does not. If $\cG'$ is another $\Fn$-basis of $\cC$, then ${\cG'}={\cG}A$ for some $A\in \GL(k,q^n)$, and hence $U_{\cG'}=U_{\cG}A$, which leaves the $\fq$-dimension of $U_{\cG}$ fixed.  
\end{remark}

\begin{remark}
 The definition of effective length and nondegeneracy of a code $\cC$ in $\cL_{n,q}[\underline{X}]$ are equivalent to those for $\Fn$-linear rank-metric codes in $V(nm,q^n)$. Indeed, let us fix $\cG$ to be an $\Fn$-basis of $\cC$  and take $\cB$ as an $\fq$-basis of $(\Fn)^m$ of the form \eqref{eq:basis}. Then, a basis of $\ev_{\cB}(\cC)$ is given by $\ev_{\cB}(\cG)$, and if we put these vectors as the rows of a generator matrix $G$, we then have that  the $\fq$-span of the columns of $G$ is exactly $U_{\cG}$. Thus, this coincides with the notion of effective length and nondegeneracy of $\Fn$-linear rank metric codes in $V(nm,q^n)$; see e.g. \cite{ABNR22}.
\end{remark}

\begin{prop}\label{prop:nondegenerate}
Let $\cC\subseteq \cL_{n,q}[\underline{X}]$ be an $\Fn$-linear rank-metric code. The following are equivalent.
 \begin{enumerate}
\item  $\cC$ is nondegenerate.
\item For any $\Fn$-basis $(g_1,\ldots, g_k)$  of $\cC$, it holds that $$ \bigcap_{i=1}^k\ker(g_i)=\{0\}.$$
\item $$ \bigcap_{f\in\cC}\ker(f)=\{0\}.$$ 
 \item $ \dd_{\rk}(\cC^\perp)>1$.
 \end{enumerate}
\end{prop}

\begin{proof}
\underline{$(b)\Longleftrightarrow (c)$:} Clear.

\underline{$(a)\Longleftrightarrow (b)$:} Consider the $\fq$-linear map
$$\begin{array}{rccl}\psi_\cG:&(\Fn)^m &\longrightarrow & (\Fn)^k \\
&v & \longmapsto & (g_1(v),\ldots,g_k(v)).
\end{array}$$
Then, by the rank-nullity theorem we have 
$$ \dim_{\fq}(\mathrm{im}(\psi_\cG))+\dim_{\fq}(\ker(\psi_{\cG}))=\dim_{\fq}(U_\cG)+\dim_{\fq}\Big(\bigcap_i\ker(g_i)\Big)=nm,$$
from which we derive the equivalence.

 \underline{$(c)\Longleftrightarrow (d)$:}  Let $h\in \cC^\perp\setminus\{0\}$. By Lemma \ref{lem:rankone}, $h$ has rank one if and only if  $h=\alpha \Tr_{q^n/q}(v \underline{X}^\top)$. Furthermore, 
 for every $f\in \cC$ we have
 $$0=f\star h=\alpha f(v).$$
 Hence, there exists $h\in\cC^\perp$ of rank one if and only if there exists a nonzero $v\in\bigcap_{f\in\cC} \ker(f)$.
 \end{proof}

\subsection{Scattered and evasive subspaces}

In this section we recall the notion of evasiveness and scatteredness of subspaces in $V(k,q^n)$, and how they are related to rank-metric codes.

\begin{definition}
Let $h,r,k,n$ be positive integers, such that $h<k$ and $h \le r$. An $\fq$-subspace $U\subseteq V(k,q^n)$ is said to be \textbf{$(h,r)$-evasive} if for every $h$-dimensional $\Fn$-subspace  $H\subseteq V(k,q^n)$, it holds $\dim_{\fq}(U\cap H)\leq r$. When $h=r$, an  $(h,h)$-evasive subspace is called \textbf{$h$-scattered}. Furthermore, when $h=1$,  a $1$-scattered subspace is simply called \textbf{scattered}.
\end{definition}

Scattered subspaces were originally introduced by Blokhuis and Lavrauw in \cite{blokhuis2000scattered}. They were later generalized for every $h$ in \cite{csajbok2021generalising}. The more general notion of evasive subspaces was instead introduced in \cite{bartoli2021evasive}, although similar notions can be found in  \cite{pudlak2004pseudorandom,guruswami2011linear,dvir2012subspace,guruswami2016explicit}. 

For what concerns $h$-scattered subspaces, there is a well-known bound on their $\fq$-dimension. Namely, an $h$-scattered subspace $U\subseteq V(k,q^n)$ satisfies
\begin{equation}\label{eq:scattered_bound}\dim_{\fq}(U)\leq \frac{kn}{h+1};\end{equation}
see \cite{blokhuis2000scattered,csajbok2021generalising}. An $h$-scattered subspace meeting \eqref{eq:scattered_bound} with equality is called a \textbf{maximum $h$-scattered subspace}.

Without loss of generality, we can restrict to study only $\fq$-subspaces $U\subseteq V(k,q^n)$ such that $\langle U\rangle_{\Fn}=V(k,q^n)$. Indeed, if this is not the case, there exists an $\Fn$-hyperplane $H\cong V(k-1,q^n)$ containing $U$, and hence we can restrict to study $U$ as an $\fq$-subspace of $V(k-1,q^n)$. Thus, from now on, we will always assume that $\langle U \rangle_{\Fn}=V(k,q^n)$.

With this assumption,  there is a natural one-to-one correspondence between $\GL(r,q)$-equivalence classes of
$k$-dimensional $\Fn$-linear rank-metric codes in $V(r,q^n)$ and $\GL(k,q^n)$-equivalence classes of $\fq$-subspaces of $V(k,q^n)$ of $\fq$-dimension $r$. This was developed in \cite{randrianarisoa2020geometric}; see also \cite{ABNR22}. Here, we rephrase it in terms of $\Fn$-linear rank-metric codes  in $\cL_{n,q}[\underline{X}]$. 

We first start defining the $\GL(nm,q)$-equivalence in this framework. Fix an $\fq$-basis $(\beta_1,\ldots,\beta_n)$ of $\Fn$. Then every $f\in\cL_{n,q}[\underline{X}]$, considered as an element of  $\Hom_{\fq}((\fq)^{nm},\Fn)$, can be written as  
$$f\Big(\sum_{j}\beta_jX_{1,j},\ldots, \sum_{j}\beta_jX_{m,j}\Big)=\tilde{f}(\beta_jX_{i,j})_{\substack{1\leq i \leq m,\\ 1 \leq j \leq n}}.$$
In this way, we can easily observe that $\GL(nm,q)$ naturally acts on $\tilde{f}$ and  thus induces an action on $\cL_{n,q}[\underline{X}]$ which preserves the rank. 

Let $\mathfrak  U(nm,k)_{q^n/q}$ denote the set of $\GL(k,q^n)$-equivalence classes $[U]$ of $nm$-dimensional $\fq$-subspaces of $V(k,q^n)$, and let $\mathfrak  C(nm,k)_{q^n/q}$ denote the set of $\GL(nm,q)$-equivalence classes $[\cC]$ of nondegenerate $k$-dimensional $\Fn$-linear codes in $\cL_{n,q}[\underline{X}]$. One can define the maps
 $$\begin{array}{rccc}\Phi: & \mathfrak  C(nm,k)_{q^n/q} &\longrightarrow &\mathfrak  U(nm,k)_{q^n/q}\\
 & [\langle g_1,\ldots,g_k\rangle_{\Fn}] & \longmapsto & [U_{\mathcal G}] \end{array}, $$
where $\mathcal G=(g_1,\ldots,g_k)$,
and 
 $$\begin{array}{rccc}\Psi: & \mathfrak  U(nm,k)_{q^n/q} &\longrightarrow &\mathfrak  C(nm,k)_{q^n/q} \\
 & [\langle u_1, \ldots, u_{nm}\rangle_{\fq}] & \longmapsto & [\ev_{\cB}^{-1}(\mathrm{rowsp}( u_1^\top \mid \ldots \mid u_{nm}^\top))] \end{array}.$$

Note that, the map $\Psi$ does not depend on the choice of the basis $\cB$, since any other $\fq$-basis $\cB'$ of $(\Fn)^m$ can be obtained via the action of $\GL(nm,q)$, and hence it gives an equivalent code.

 \begin{theorem}[\textnormal{\cite{randrianarisoa2020geometric}}]\label{thm:correspondence_codes_systems}
  The maps $\Phi$ and $\Psi$ are well-defined and they are one the inverse of the other. Hence, they define a one-to-one correspondence between equivalence classes of nondegenerate $k$-dimensional $\Fn$-linear codes in $\cL_{n,q}[\underline{X}]$ and equivalence classes of $\fq$-subspaces of $V(k,q^n)$ of $\fq$-dimension $nm$.
 \end{theorem}

The correspondence in Theorem \ref{thm:correspondence_codes_systems} induces a correspondence between maximum $h$-scattered subspaces and MRD codes. We reformulate it in our setting, while the more general version can be found in \cite[Theorem 3.2]{zini2021scattered}; see also \cite[Theorem 4.9]{marino2022evasive}.

\begin{theorem}[\textnormal{\cite[Theorem 3.2]{zini2021scattered}}]\label{thm:hscattered_MRD}
 Suppose that $h+1$ divides $k$ and let $m:=\frac{k}{h+1}$. Let $U$ be an $nm$-dimensional $\fq$-subspace in $V(k,q^n)$ and let $\cC\in \Psi([U])$ be any of its associated $k$-dimensional $\Fn$-linear rank-metric codes in $\cL_{n,q}[\underline{X}]$. Then, $U$ is maximum $h$-scattered if and only if $\cC$ is an MRD code.
\end{theorem}

We conclude by remarking the fact that the setting of $\cL_{n,q}[\underline{X}]$ is a bit more restrictive for studying scattered subspaces and MRD codes, since we are fixing the dimension of the $\fq$-subspaces to be a multiple of $n$ -- or in other words, we are fixing the size of the matrices to be one multiple of the other. 
However, in this way, we will see that we can  take advantage of the multivariate polynomial representation, using tools described in  Section \ref{sec:curves} in order to derive new construction of maximum scattered subspaces -- and hence MRD codes.

\section{Indecomposable $h$-scattered sequences}\label{sec:scattered_sequences}

We start with this definition.
\begin{definition}\label{Def:ScatteredSequence}
Let $\mathcal I:=(i_1, i_2,\ldots, i_m) \in (\Z/n\Z)^m$ and consider $f_1,\ldots,f_s\in \cL_{n,q}[\underline{X}]$. We define the \textbf{$\mathcal I$-space}
$U_{\mathcal I,\mathcal{F}}:=U_{\mathcal{F}'},$
where
$$\mathcal F'=(X_1^{q^{i_1}},\ldots, X_m^{q^{i_m}}, f_1,\ldots, f_s).$$
The $s$-tuple $\mathcal{F}:=(f_1,\ldots,f_s)$ is said to be an \textbf{$(\mathcal I;h)_{q^n}$-scattered sequence of order $m$} if the $\mathcal I$-space
$U_{\mathcal I,\mathcal{F}}$
is maximum $h$-scattered in $V(m+s,q^{n})$, 
\end{definition}

Note that for $m=1$ and $h=1$ the above definition coincides with the one of scattered polynomials as in \cite{sheekey2016new}. In particular, $h$-scattered sequences with $h=1$ will be simply called scattered sequences.

\begin{definition}\label{Def:ExceptionalScatteredSequence}
An $(\mathcal I;h)_{q^n}$-scattered sequence  $\mathcal{F}:=(f_1,\ldots,f_s)$ of order $m$  is said to be \textbf{exceptional} if it is $h$-scattered over infinitely many extensions $\mathbb{F}_{q^{n\ell}}$ of $\mathbb{F}_{q^n}$.
\end{definition}

We consider the natural operation of direct sum on subspaces of $V(k_1,q^n)$ and $V(k_2,q^n)$ whose dimension is multiple of $n$. This can be identified with the operation on sequences of multivariate linearized polynomials obtained by juxtaposing the two corresponding sequences. Indeed, let $f_1,\ldots,f_{k_1} \in \cL_{n,q}[X_1,\ldots,X_m]$ and $g_1,\ldots,g_{k_2}\in \cL_{n,q}[X_1,\ldots,X_{m'}]$. Let $\mathcal F \oplus \mathcal G:=(f_1,\ldots, f_{k_1},g_1,\ldots,g_{k_2})\in \cL_{n,q}[X_1,\ldots,X_{m+m'}]$, then it is immediate to see that
$$ U_{\mathcal F} \oplus U_{\mathcal G} = U_{\mathcal{F}\oplus \mathcal{G}}.$$
When dealing with $nm$-dimensional $\fq$-subspaces of $V(k,q^n)$, they can all be represented by spaces of the form $U_{\mathcal F}$, for $\mathcal F=(f_1,\ldots,f_k)$. Thus, we can give the following definition.

\begin{definition}
 An $nm$-dimensional $\fq$-subspace $U$ of $V(k,q^n)$   is said to be \textbf{decomposable} if it can be written as
 $$U=U_{\mathcal F}\oplus U_{\mathcal G}$$
 for some nonempty $\mathcal F, \mathcal G$. 
 When this happens we say that $\mathcal{F}$ and $\mathcal{G}$ are \textbf{factors} of $\mathcal{H}$. Furthermore, 
 $U$ is then said to be \textbf{indecomposable} if it is not decomposable.
\end{definition}

Let us now consider the direct sum of $h$-scattered sequences.
Let $\mathcal I:=(i_1,\ldots,i_m)$, $\mathcal J:=(j_1,\ldots,j_{m^{\prime}})$, let $\mathcal{F}=(f_1,\ldots,f_{s})$ and $\mathcal{G}=(g_1,\ldots,g_{s^{\prime}})$ be $(\mathcal I;h)_{{q^n}}$ and $(\mathcal J;h)_{{q^n}}$-scattered sequences of orders $m$ and $m^{\prime}$, respectively.
The {direct sum} $\mathcal{H}:=\mathcal{F}\oplus \mathcal{G}$ is the $(s+s^{\prime})$-tuple $(f_1,\ldots,f_s,g_1,\ldots,g_{s^{\prime}})$. Since 
$$U_{\mathcal I\oplus\mathcal J,\mathcal{H}}=U_{\mathcal I,\mathcal{F}}\oplus U_{\mathcal J,\mathcal{G}},$$
$\mathcal{H}$ is an $(\mathcal I \oplus \mathcal J;h)_{q^{n}}$-scattered sequence of order $m+m^{\prime}$; see \cite{BGMP}. 

For any $m\geq 1$, there exist many $h$-scattered sequences of order $m$ obtained as direct sums of scattered polynomials and thus it is  natural to search for examples of $h$-scattered sequences which cannot be obtained as direct sums.

\begin{lemma}\label{Lemma:ind}
Let  $\mathcal{F}:=(f_1,\ldots,f_s)$ be an exceptional $(\mathcal I;h)_{q^n}$-scattered sequence of order $m$. If $U_{\mathcal I,\mathcal{F}}$ is $(r,rn/(h+1)-1)$-evasive for any $r\in [h+1,\lfloor (m+s)/2\rfloor ]$ with $(h+1)\mid rn$ then $\mathcal{F}$ is indecomposable.
\end{lemma}
\begin{proof}
Let $r\in [h+1,\lfloor (m+s)/2\rfloor ]$. A maximum $h$-scattered linear set in $V(r,q^n)$ has dimension $rn/(h+1)$. If $\mathcal{F}$ has a factor of   order $r$ then $\dim_{\fq}(U_{\mathcal I,\mathcal{F}}\cap V(r,q^n))=rn/(h+1)$, a contradiction to the   $(r,rn/(h+1)-1)$-evasiveness. 
\end{proof}

\subsection{Indecomposable $h$-scattered sequences and ordinary duality}

Let $\sigma: V\times V\longrightarrow \Fn$ be a nondegenerate
	bilinear form on $V=V(r,q^n)$ and define
	\[ \begin{array}{rccl}
		\sigma' \colon & V\times V &\longrightarrow & \F_q,\\
       &({u}, { v}) 	& \longmapsto &	\Tr_{q^n/q}(\sigma({u}, {v})).
	\end{array} \]
	Then $\sigma'$ is a nondegenerate bilinear form on
	$V$, when $V$ is  regarded as an $rn$-dimensional vector space  over
	$\fq$. Let $\tau$ and $\tau'$ be the orthogonal complement maps
	defined by $\sigma$ and $\sigma'$ on the lattices  of the
	$\F_{q^n}$-subspaces and $\fq$-subspaces of $V$, respectively.
	Recall that  if $R$ is an
	$\Fn$-subspace of $V$ and $U$ is an $\fq$-subspace of $V$
	then $U^{\tau'}$ is an $\fq$-subspace of $V$, $\dim_{\Fn}(R^\tau)+\dim_{\Fn}(R)=r$ and
	$\dim_{\fq} (U^{\tau'})+\dim_{\fq} (U)= rn$. It easy to see
	that $R^\tau=R^{\tau'}$ for each $\Fn$-subspace $R$ of $V$. For a more detailed explanation, we refer to \cite[Chapter 7]{taylor1992geometry}.
	
	With the notation above, $U^{\tau '}$ is called the \textbf{dual} of $U$ (with respect to $\tau'$). 
	Up to $\mathrm{GL}(k,q^n)$-equivalence, the dual of an $\fq$-subspace of $V$ does not depend on the choice of the nondegenerate bilinear forms $\sigma$ and $\sigma'$ on $V$. For more details see \cite{polverino2010linear}. If $R$ is an $s$-dimensional $\Fn$-subspace of $V$ and $U$ is a $t$-dimensional $\fq$-subspace of $V$, then
	\begin{equation}\label{pesi}
		\dim_{\fq}(U^{\tau'}\cap R^\tau)-\dim_{\fq}(U\cap R)=rn-t-sn.
	\end{equation}
	
	\begin{prop}
	The dual of an indecomposable scattered subspace is an indecomposable scattered subspace as well.
	\end{prop}
	
	\begin{proof}
	Let $U$ be an indecomposable scattered subspace of $V=V(r,q^n)$. Then $rn$ is even, $\dim_{\F_q} U=rn/2$ and there exists $2\leq i\leq r/2$ such that $V=V_1\oplus V_2$, where $V_1=V(i,q^n)$, $V_2=V(r-i,q^n)$, $\dim_{\F_q}(U\cap V_1)=in/2$ and $\dim_{\F_q}(U\cap V_2)=(r-i)n/2$. Also, $V=V_1^\tau\oplus V_2^\tau$, $U^{\tau'}$ is a maximum scattered $\F_q$-subspace of $V$ and from Equation \eqref{pesi} we get
	\[\dim_{\F_q} (U^{\tau'}\cap V_1^{\tau})=\frac{in}{2}+rn-\frac{rn}{2}-in=\frac{(r-i)n}{2}\]
    and $\dim_{\F_q}(U^{\tau '}\cap V_2^\tau)=\frac{in}{2}$, i.e. $U^{\tau'}$ is indecomposable.
	\end{proof}

\subsection{Indecomposable $h$-scattered sequences and Delsarte duality}

In \cite[Section 3]{csajbok2021generalising}, another type of duality has been introduced. 	
	Let $U$ be an $m$-dimensional $\F_q$-subspace of a vector space $V=V(k,q^n)$, with $m>k$. By \cite[Theorems 1, 2]{lunardon2004translation} (see also \cite[Theorem 1]{lunardon2002geometric}), there is an embedding of $V$ in $Z=V(m,q^n)$ with $Z=V \oplus \Gamma$ for some $(m-k)$-dimensional $\F_{q^n}$-subspace $\Gamma$ such that
	$U=\langle W,\Gamma\rangle_{\F_{q}}\cap V$, where $W$ is an $m$-dimensional $\F_q$-subspace of $Z$, $\langle W\rangle_{\F_{q^n}}=Z$ and $\Gamma\cap V=W\cap \Gamma=\{{0}\}$.
	Then the quotient space $Z/\Gamma$ is isomorphic to $V$ and under this isomorphism $U$ is the image of the $\F_q$-subspace $W+\Gamma$ of $Z /\Gamma$.
	Now, let $\beta'\colon W\times W\rightarrow\F_{q}$ be a non-degenerate bilinear form on $W$. Then $\beta'$ can be extended to a non-degenerate bilinear form $\beta\colon Z\times Z\rightarrow\F_{q^n}$.
	Let $\perp$ and $\perp'$ be the orthogonal complement maps defined by $\beta$ and $\beta'$ on the lattice of $\F_{q^n}$-subspaces of $Z$ and of $\F_q$-subspaces of $W$, respectively.
	The $m$-dimensional $\F_q$-subspace $W+\Gamma^{\perp}$ of the quotient space $Z/\Gamma^{\perp}$  will be denoted by $\bar U$ and we call it the \textbf{Delsarte dual} of $U$ with respect to $\beta'$. By \cite[Remark 3.7]{csajbok2021generalising}, up to $\GL(m,q)$-equivalence,  the Delsarte dual of an $m$-dimensional $\F_q$-subspace does not depend on the choice of the nondegenerate bilinear form on $W$. 	

The following result relates the Delsarte dual of an $\fq$-subspace of $V(nm,q)$ with the dual of a rank-metric code in $\cL_{n,q}[\underline{X}]$.

\begin{theorem}\label{thm:duality_delsarte_codes}
Let $\cC\subseteq \cL_{n,q}[\underline{X}]$ be a nondegenerate $\Fn$-linear rank-metric code with $\textcolor{cyan}{\dd_{\mathrm{rk}}}(\cC)>1$, and let $U\in \Phi([\cC])$. 
Then $\Phi([\cC^\perp])=[U^\perp]$.
\end{theorem}

\begin{proof}
 The proof can be easily obtained  extending the one in \cite[Theorem 4.12]{csajbok2021generalising}, where it must be noted that the hypothesis of the code $\cC$ being MRD is not needed. Indeed, in the proof of that result, the only properties used are the fact that $\cC$ has no elements of rank one and that the elements in $\cC$ have no common nonzero elements in their kernel; see also \cite[Remark 2.19]{marino2022evasive}. By Proposition \ref{prop:nondegenerate}, this last hypothesis is equivalent to $\cC$ being nondegenerate.
\end{proof}

\begin{prop}
 The Delsarte dual of an indecomposable subspace is an indecomposable subspace.
\end{prop}

\begin{proof}
Let $U$ be an indecomposable subspace and suppose on the contrary that $U^\perp=U_1\oplus U_2$ and let $\cC \in \Psi([U])$. By Theorem \ref{thm:duality_delsarte_codes}, we have
$[\cC^\perp]= \Psi([U^\perp])=\Psi([U_1\oplus U_2])=[\cC_1\oplus \cC_2]$, where $\cC_i \in \Psi([U_i])$ for $i \in \{1,2\}$.
Thus, $\cC^\perp$ is equivalent to $\cC_1\oplus\cC_2$, which implies 
$[\cC]=[\cC_1^\perp\oplus\cC_2^\perp]$. In particular, this means
$$[U]=[U_1^\perp \oplus U_2^\perp],$$
which contradicts the hypothesis of $U$ being indecomposable.
\end{proof}

\section{The first infinite family of indecomposable exceptional scattered sequences of order larger than 1}\label{sec:construction}

A first example of indecomposable $((0,0),1)_{q^4}$-scattered sequence of order larger than one for $q=2^{2s+1}$ was provided in \cite{BMN2022} and it consists of the pair $(x^q+y^{q^2},x^{q^2}+y^q+y^{q^2})$. 
In this paper we provide a generalization of this example to an infinite family of exceptional type.

\begin{definition}
\label{defn:sets}
Let $n$ be a positive integer and consider the finite field $\mathbb{F}_{q^n}$. For each choice of $\alpha,\beta,\gamma \in \mathbb{F}_{q^n}^*$, and $I\neq J\in \mathbb{N}$, $I,J<n-1$,  we define the set
\begin{equation*}
        U_{\alpha, \beta, \gamma}^{I,J,n}:=\left\{\left(x,y,x^{q^I}+\alpha y^{q^J},x^{q^J}+\beta y^{q^I} + \gamma y^{q^J}\right) \ : \ x,y\in \F_{q^n}\right\}.
\end{equation*}
\end{definition}

We can immediately give the following result which gives a sufficient condition on $U_{\alpha, \beta, \gamma}^{I,J,n}$ for being exceptional scattered.

\begin{theorem}
\label{thm:scattered:1}
Assume that $\mathrm{gcd}(I,J,n)=1$ and that the polynomial 
\begin{equation}\label{Eq:P}
    P^{I,J}_{\alpha,\beta,\gamma}(X):=\begin{cases}
    X^{q^{J-I}+1} + \gamma X - \alpha\beta,& \mathrm{if } I<J,\\ 
     X^{q^{I-J}+1} + \gamma X^{q^{I-J}} - \alpha\beta,& \mathrm{if } I>J,\\
    \end{cases}
\end{equation}
has no roots in $\F_{q^n}$. Then the set $U_{\alpha, \beta, \gamma}^{I,J,n}$ is exceptional scattered.
\end{theorem}

\begin{proof}
Assume that $P^{I,J}_{\alpha,\beta,\gamma}(X)$ has no roots in $\F_{q^n}$ and let $\lambda \in \mathbb{F}_{q^n}\setminus \F_q$ be such that 
\begin{equation}
\label{eq:0}
    \left(x,y,x^{q^I}+\alpha y^{q^J},x^{q^J}+\beta y^{q^I} + \gamma y^{q^J}\right) = \lambda \left(u,v,u^{q^I}+\alpha v^{q^J},u^{q^J}+\beta v^{q^I} + \gamma v^{q^J}\right),
\end{equation}
with $x,y,u,v \in \mathbb{F}_{q^n}$.
The set $U_{\alpha, \beta, \gamma}^{I,J,n}$ is maximum scattered if and only if the previous equation holds only for $u=v=0$. 

By way of contradiction, we assume that $(u,v)\neq (0,0)$.  We have 
\begin{equation}
\label{eq:system:scattered}
    \begin{cases}
    x=\lambda u\\
    y = \lambda v\\
    \lambda^{q^I}u^{q^I} + \alpha\lambda^{q^J}v^{q^J} = \lambda\left(u^{q^I}+\alpha v^{q^J}\right)\\
    \lambda^{q^J}u^{q^J} + \beta\lambda^{q^I}v^{q^I} + \gamma\lambda^{q^J}v^{q^J} = \lambda\left(u^{q^J}+\beta v^{q^I} + \gamma v^{q^J}\right).
    \end{cases}
\end{equation}
The last two equations in \eqref{eq:system:scattered} can be rewritten as 
\begin{equation}
\label{eq:1}
  \lambda^{q^J}\alpha v^{q^J} + \lambda^{q^I}u^{q^I} - \lambda\left(u^{q^I}+\alpha v^{q^J}\right) = 0 
\end{equation}
and
\begin{equation}
\label{eq:2}
 \lambda^{q^J}\left(u^{q^J} + \gamma v^{q^J}\right) + \lambda^{q^I}\beta v^{q^I} - \lambda\left(u^{q^J}+\beta v^{q^I} + \gamma v^{q^J}\right) = 0.
\end{equation}
Multiplying \eqref{eq:1} by $\left(u^{q^J} + \gamma v^{q^J}\right)$ and \eqref{eq:2} by $\alpha v^{q^J}$, and taking the difference of the obtained equations, we have 
\begin{equation}
\label{eq:3}
    \left(\lambda^{q^I} - \lambda\right)\left(u^{q^I + q^J} + \gamma u^{q^I}v^{q^J} - \alpha\beta v^{q^I + q^J}\right) = 0.
\end{equation}
If $v=0$ then $u\neq 0$ and $\lambda^{q^I} - \lambda=0$, i.e. $\lambda \in \mathbb{F}_{q^I}$. If $v\neq 0$, letting $X:=\frac{u^{q^I}}{v^{q^I}}$ (if $I < J$) or $X:=\frac{u^{q^J}}{v^{q^J}}$ (if $I > J$), we can rewrite $ (u^{q^I + q^J} + \gamma u^{q^I}v^{q^J} - \alpha\beta v^{q^I + q^J} )/v^{q^I+q^J}$ as $P^{I,J}_{\alpha,\beta,\gamma}(X)$.
By assumption, the  polynomial $P^{I,J}_{\alpha,\beta,\gamma}(X)$ has no roots in $\F_{q^n}$, hence \eqref{eq:3} is satisfied if and only if $\lambda\in \F_{q^I}$.

We consider now the difference between \eqref{eq:1} multiplied by $\beta v^{q^I}$ and \eqref{eq:2} multiplied by $u^{q^I}$ and we have
\begin{equation}
\label{eq:4}
    \left(\lambda^{q^J} - \lambda\right)\left(\alpha\beta v^{q^I + q^J} -\gamma u^{q^I}v^{q^J} - u^{q^I + q^J}\right) = 0.
\end{equation}
Then, arguing as above we see that Equation \eqref{eq:4} is satisfied if and only if $\lambda\in \F_{q^J}$.

We have therefore obtained that the values of $\lambda$ satisfying \eqref{eq:0}
need to be $\lambda\in \F_{q^I}\cap \F_{q^J}\cap \F_{q^n}$. As, by assumption, $\mathrm{gcd}(I,J,n)=1$, we hence have that $\lambda\in \F_q$, a contradiction. So $(u,v)=(0,0)$, which yields that the set $U_{\alpha, \beta, \gamma}^{I,J,n}$ is scattered.

The fact that $U_{\alpha, \beta, \gamma}^{I,J,n}$ is exceptional scattered follows directly from the discussion above. Indeed, let $\F_{q^{n\ell}}$ be the extension field of $\F_{q^n}$ that is the splitting field of the polynomial $P(X)$. Then, there exist infinitely many integers $t$ satisfying the following two conditions:
\begin{itemize}
    \item $\mathrm{gcd}(I,J,nt)=1$,
    \item the polynomial $P^{I,J}_{\alpha,\beta,\gamma}(X)$ has no roots in $\F_{q^{nt}}$.
\end{itemize}
 This can be seen as all the $t\in \mathbb{N}$ such that $\mathrm{gcd}(I,J,t)=1$ and $\mathrm{gcd}(\ell,t)=1$ are suitable. Hence, the set $U_{\alpha, \beta, \gamma}^{I,J,nt} = \left\{\left(x,y,x^{q^I}+\alpha y^{q^J},x^{q^J}+\beta y^{q^I} + \gamma y^{q^J}\right) \ : \ x,y\in \F_{q^{nt}}\right\}$ is scattered for infinitely many $t$, meaning that $U_{\alpha, \beta, \gamma}^{I,J,n}$ is exceptional scattered. 
\end{proof}
Note that for $\alpha=\beta=\gamma=1$, $I=1$, $J=2$,
one obtains the indecomposable maximum scattered linear set in \cite{BMN2022}.

\begin{remark}\label{rem:no_roots}
Apart from the maximum scattered subspaces found in \cite{BMN2022}, we want to point out that in this paper we provide many more constructions, and the family that we propose is nonempty for infinitely many $n$ and $q$. To see this, we just need to prove that we can always choose $\alpha, \beta,\gamma$ such that the polynomial $P^{I,J}_{\alpha,\beta,\gamma}(X)$ has no roots in $\Fn$. If we restrict to the case that $K=J-I>0$ and $n$ are coprime, then the polynomial $P^{I,J}_{\alpha,\beta,\gamma}(X)$ is a projective polynomial associated to the automorphism $\sigma:x\longmapsto x^{q^K}$. The  linearized polynomial associated with $P^{I,J}_{\alpha,\beta,\gamma}(X)$ is $$f(X):=X^{q^{2K}} + \gamma X^{q^K} - \alpha\beta X.$$
By \cite[Theorem 6]{mcguire2019characterization}, $P^{I,J}_{\alpha,\beta,\gamma}(X)$ has no roots in $\Fn$ if and only if the matrix $A_f$ has no eigenvalues in $\fq$, where
$$A_f:=C_f C_f^\sigma\cdot\ldots\cdot C_f^{\sigma^{n-1}},$$
and 
$$C_f=\begin{pmatrix}
0 & \alpha\beta  \\
1& -\gamma
\end{pmatrix}$$
is the companion matrix associated with $f$. 

We can choose $\alpha\beta,\gamma \in \fq^*$ such that the corresponding degree $2$ polynomial $\tilde{f}:=X^2+\gamma X-\alpha\beta$ associated with $f$ is a primitive polynomial, that is, it is irreducible and its roots $\eta_1,\eta_2$ are generators of $\F_{q^2}^*$. Thus, since the coefficients are in $\fq$,  $A_f=(C_f)^n$, and its eigenvalues are $\eta_1^n,\eta_2^n$. If $n \not\equiv 0 \pmod{q+1}$, then $\eta_1^n,\eta_2^n\notin \fq^*$ and the polynomial $P^{I,J}_{\alpha,\beta,\gamma}(X)$ has no roots in $\Fn$. 

Since there are $\varphi(q^2-1)/2$, where $\varphi$ is the Euler's totient function, primitive polynomials $\tilde{f}$ of degree $2$ and $\alpha,\beta\in \mathbb{F}_{q^n}^*$, this shows that there are at least $(q^n-1)\varphi(q^2-1)/2$ choices for $P^{I,J}_{\alpha,\beta,\gamma}(X)$ with $\gcd(K,n)=1$ and $(q+1)\nmid n$.
\end{remark}

Using algebraic curves  over finite fields we can actually prove the converse of Theorem \ref{thm:scattered:1} in a small-degree regime for $I$ and $J$. 

\begin{theorem}
\label{thm:scattered:2}
Assume that $\mathrm{gcd}(I,J,n)=1$ and $\max\{I,J\}\leq n/4$.
If the set $U_{\alpha, \beta, \gamma}^{I,J,n}$ is scattered, then 
the polynomial 
\begin{equation}
    P^{I,J}_{\alpha,\beta,\gamma}(X):=\begin{cases}
    X^{q^{J-I}+1} + \gamma X - \alpha\beta,& \mathrm{if } I<J,\\ 
     X^{q^{I-J}+1} + \gamma X^{q^{I-J}} - \alpha\beta,& \mathrm{if } I>J,\\
    \end{cases}
\end{equation}
has no roots in $\F_{q^n}$.
\end{theorem}

\begin{proof}
Assume that $P^{I,J}_{\alpha,\beta,\gamma}(X)$ has a root $\mu \in \F_{q^n}$. Then, with the notations as in the proof of Theorem \ref{thm:scattered:1}, $\frac{u^{q^I}}{v^{q^I}}=\mu$ (if $I < J$) or $\frac{u^{q^J}}{v^{q^J}}=\mu$ (if $I > J$), and we let $\tilde{\mu}:=\frac{1}{\mu^{q^{-I}}}=\frac{v}{u}$ (if $I < J$) or $\tilde{\mu}:=\frac{1}{\mu^{q^{-J}}}=\frac{v}{u}$ (if $I > J$). In this way, we can rewrite the third equation in \eqref{eq:system:scattered} as
\begin{equation*}
    u^{q^I}\left(\lambda^{q^I}-\lambda\right) + \alpha\tilde{\mu}^{q^J}u^{q^J}\left(\lambda^{q^J}-\lambda\right) = 0.
\end{equation*}
Note that this equation defines the reducible curve 
\begin{equation*}
    \mathcal{X}: u^{q^I}\left(\lambda^{q^I}-\lambda\right) + \alpha\tilde{\mu}^{q^J}u^{q^J}\left(\lambda^{q^J}-\lambda\right) = 0
\end{equation*}
in $\mathbb{A}^2(\F_{q^n})$, with coordinates $(u,\lambda)$. Then, there are two possible cases:

\begin{enumerate}
    \item if $I < J$, let $K:=J - I$. The curve
\begin{equation}
\label{eq:curve:1}
    \mathcal{Y}: -\alpha \tilde{\mu}^{q^J}u^{q^J-q^I} = \frac{\prod_{\vartheta\in \F_{q^I}\setminus \F_{q^{\mathrm{gcd}(I,J)}}}\left(\lambda - \vartheta\right)}{\prod_{\eta\in \F_{q^J}\setminus \F_{q^{\mathrm{gcd}(I,J)}}}\left(\lambda - \eta\right)} 
\end{equation}
is an $\F_{q^n}$-rational component of $\mathcal{X}$.
Note that, applying the Frobenius automorphism to \eqref{eq:curve:1}, it is possible to see that $\mathcal{Y}$ is also defined by the following equation
\begin{equation}
    \label{eq:curve:frobenius:1}
    u^{q^K-1} = \frac{1}{A} \frac{\prod_{\vartheta\in \F_{q^I}\setminus \F_{q^{\mathrm{gcd}(I,J)}}}\left(\lambda - \vartheta\right)}{\prod_{\eta\in \F_{q^J}\setminus \F_{q^{\mathrm{gcd}(I,J)}}}\left(\lambda - \eta\right)}, 
\end{equation}
where $A:=\left(-\alpha\tilde{\mu}^{q^J}\right)^{q^{-I}}$.

\item If instead $I > J$, let $K:=I - J$. The curve
\begin{equation}
\label{eq:curve:2}
    \mathcal{Y}: u^{q^I-q^J} = -\alpha \tilde{\mu}^{q^J}\frac{\prod_{\eta\in \F_{q^J}\setminus \F_{q^{\mathrm{gcd}(I,J)}}}\left(\lambda - \eta\right)}{\prod_{\vartheta\in \F_{q^I}\setminus \F_{q^{\mathrm{gcd}(I,J)}}}\left(\lambda - \vartheta\right)} 
\end{equation}
is an $\F_{q^n}$-rational component of $\mathcal{X}$.
As above, applying the Frobenius automorphism to \eqref{eq:curve:2}, it is possible to see that $\mathcal{Y}$ is also defined by the following equation
\begin{equation}
    \label{eq:curve:frobenius:2}
    u^{q^K-1} = B \frac{\prod_{\eta\in \F_{q^J}\setminus \F_{q^{\mathrm{gcd}(I,J)}}}\left(\lambda - \eta\right)}{\prod_{\vartheta\in \F_{q^I}\setminus \F_{q^{\mathrm{gcd}(I,J)}}}\left(\lambda - \vartheta\right)}, 
\end{equation}
where $B:=\left(-\alpha\tilde{\mu}^{q^J}\right)^{q^{-J}}$.
\end{enumerate}

In case (a) (resp. (b)), as $\mathrm{gcd}(q^K-1,q)=1$ and $\frac{\prod_{\vartheta\in \F_{q^I}\setminus \F_{q^{\mathrm{gcd}(I,J)}}}\left(\lambda - \vartheta\right)}{\prod_{\eta\in \F_{q^J}\setminus \F_{q^{\mathrm{gcd}(I,J)}}}\left(\lambda - \eta\right)}$ has valuation either 0, 1 or $-1$ at all the places of $\mathbb{P}^1(\overline{\F}_{q^n})$, except possibly at the place at infinity, we have that the projective closure $\overline{\mathcal{Y}}$ of \eqref{eq:curve:frobenius:1} (resp. \eqref{eq:curve:frobenius:2}) is a Kummer cover of $\mathbb{P}^1(\F_{q^n})$.

Hence, we can readily compute the genus $g'$ of $\overline{\mathcal{Y}}$ following \cite[Corollary 3.7.4]{Stichtenoth}:
\begin{equation*}
\begin{split}
    g' &= 1 - (q^K-1) + \frac{1}{2}\left(q^K - 2\right)\left(q^I + q^J - 2q^{\mathrm{gcd}(I,J)}\right)\\
    &= \frac{q^{\mathrm{max}\{I,J\}+K}}{2} + G(q),
\end{split}
\end{equation*}
where $G(q)$ is a polynomial in $q$ of degree  $\mathrm{max}\{I,J\}$. Note that there are at most $(q+2)\deg(\mathcal{Y})\leq (q+2)(q^K-1)(q^I-q^{\gcd(I,J)})$ places centered at points on the line at infinity or on $(\lambda^q-\lambda)u=0$.
By the Hasse-Weil bound of Theorem \ref{thm:hasse-weil}, we hence have that 
\begin{equation*}
    |\overline{\mathcal{Y}}(\F_{q^n})| \geq q^n +1 -2g'q^{\frac{n}{2}} -(q+2)(q^K-1)(q^I-q^{\gcd(I,J)})>1
\end{equation*}
as, by assumption, $\max\{I,J\}\leq n/4$ and thus there exists  $\lambda\in \F_{q^n}\setminus\mathbb{F}_q$ such that \eqref{eq:0} is satisfied for non-zero values of $u,v$ and hence the set $U_{\alpha, \beta, \gamma}^{I,J,n}$ is not scattered.
\end{proof}

To prove the indecomposability of $U_{\alpha, \beta, \gamma}^{I,J,n}$ the following result will be crucial. 
\begin{theorem}
\label{thm:evasive}
If  $P^{I,J}_{\alpha,\beta,\gamma}(X)$ has no roots in $\F_{q^n}$ then $U_{\alpha, \beta, \gamma}^{I,J,n}$ is $(2,2\max\{I,J\})_q $-evasive. 
\end{theorem}

\begin{proof}
To prove that $U_{\alpha, \beta, \gamma}^{I,J,n}$ is $(2,2\max\{I,J\})_q $-evasive, we need to show that any $\F_{q^n}$-subspace of dimension 2 contains at most $q^{2\max\{I,J\}}$ vectors of $U_{\alpha, \beta, \gamma}^{I,J,n}$.

Let $w_1:=\left(x,y,x^{q^I}+\alpha y^{q^J},x^{q^J}+\beta y^{q^I} + \gamma y^{q^J}\right)$ and $w_2:=\left(z,t,z^{q^I}+\alpha t^{q^J},z^{q^J}+\beta t^{q^I} + \gamma t^{q^J}\right)$ be two vectors in $U_{\alpha, \beta, \gamma}^{I,J,n}$ that are $\F_{q^n}$-independent. A vector $w_3:=\left(u,v,u^{q^I}+\alpha v^{q^J},u^{q^J}+\beta v^{q^I} + \gamma v^{q^J}\right)$ lies in $\langle w_1,w_2\rangle_{\F_{q^n}}$ if and only if the following matrix 
\begin{equation*}
    \mathfrak{M}:=\begin{pmatrix}
    x & y & x^{q^I}+\alpha y^{q^J} & x^{q^J}+\beta y^{q^I} + \gamma y^{q^J}\\
    z & t & z^{q^I}+\alpha t^{q^J} & z^{q^J}+\beta t^{q^I} + \gamma t^{q^J}\\
    u & v & u^{q^I}+\alpha v^{q^J} & u^{q^J}+\beta v^{q^I} + \gamma v^{q^J}\\
    \end{pmatrix}
\end{equation*}  has rank $2$.

We now study the number of $(u,v)\in \Fn^2$ such that $\rk(\mathfrak{M})=2$, by imposing the $3\times 3$ minors of $\mathfrak{M}$ to be zero.

If $xt-yz\neq 0$, this is equivalent to determining the number of solutions of the following system
\begin{numcases}{}
    u^{q^I}+\alpha v^{q^J} - Av + Bu = 0 \label{eq:orlati:1}\\
    u^{q^J}+\beta v^{q^I} + \gamma v^{q^J} - Cv + Du = 0, \label{eq:orlati:2}
\end{numcases}
where
\begin{equation*}
\begin{split}
    A &:= \frac{x \left(z^{q^I}+\alpha t^{q^J}\right) - z\left(x^{q^I}+\alpha y^{q^J}\right)}{xt-yz},\\
    B &:= \frac{y \left(z^{q^I}+\alpha t^{q^J}\right) - t\left(x^{q^I}+\alpha y^{q^J}\right)}{xt-yz},\\
    C &:= \frac{x \left(z^{q^J}+\beta t^{q^I} + \gamma t^{q^J}\right) - z\left(x^{q^J}+\beta y^{q^I} + \gamma y^{q^J}\right)}{xt-yz},\\
    D &:= \frac{y \left(z^{q^J}+\beta t^{q^I} + \gamma t^{q^J}\right) - t\left(x^{q^J}+\beta y^{q^I} + \gamma y^{q^J}\right)}{xt-yz}.
\end{split}
\end{equation*}
Note that \eqref{eq:orlati:1} and \eqref{eq:orlati:2} define two plane curves
\begin{equation*}
    \begin{split}
        \mathcal{X}_1 &: u^{q^I}+\alpha v^{q^J} - Av + Bu = 0, \\
        \mathcal{X}_2 &: u^{q^J}+\beta v^{q^I} + \gamma v^{q^J} - Cv + Du = 0
    \end{split}
\end{equation*}
in $\mathbb{A}^2(\F_{q^n})$, with coordinates $(u , v)$.
Hence, we can estimate the number of solutions of the previous system by estimating the number of intersections of such curves. To this aim, in order to use Bézout's theorem, we first show that $\mathcal{X}_1$ and $\mathcal{X}_2$ have no common components. Consider the projective closures $\overline{\mathcal{X}_1}$ and $\overline{\mathcal{X}_2}$ of the curves in $\mathbb{P}^2(\F_{q^n})$, with coordinates $[u : v : w]$ and $r: w=0$ being the line at infinity. 

\begin{itemize}
    \item Suppose that $I<J$. Then $\overline{\mathcal{X}_1}\cap r=\{[1 : 0 : 0]\}$, while $\overline{\mathcal{X}_2}\cap r=\{[-\gamma^{q^{-J}} : 1 : 0]\}$. 
    \item Suppose that $I>J$. Then $\overline{\mathcal{X}_1}\cap r=\{[0 : 1 : 0]\}$, while $\overline{\mathcal{X}_2}\cap r=\{[1: 0 : 0]\}$. 
\end{itemize}

In both  cases, as the curves intersect the same line $r$ in two different points, we conclude that they cannot have a common component, otherwise we would find the points of intersection of such a component with $r$ appearing in $\left(\overline{\mathcal{X}_1}\cap r\right) \cap \left(\overline{\mathcal{X}_2}\cap r\right)$, which we have shown to be empty. \\  Then, by Bézout's Theorem, we have that the number of solutions of the system defined by \eqref{eq:orlati:1} and \eqref{eq:orlati:2} is at most $q^{2\max\{I,J\}}$.

If instead $xt-yz=0$, we distinguish a number of cases.
\begin{enumerate}
    \item[(i)] If $x \left(z^{q^I}+\alpha t^{q^J}\right) - z\left(x^{q^I}+\alpha y^{q^J}\right)\neq 0$, then the $2\times 2$ submatrix of $\mathfrak{M}$ given by 
\begin{equation*}
    \begin{pmatrix}
    x &  x^{q^I}+\alpha y^{q^J} \\
    z &  z^{q^I}+\alpha t^{q^J} \\
    \end{pmatrix}
\end{equation*}
has determinant different from zero. Hence, in this case we need to estimate the number of solutions of the following system:
\begin{numcases}{}
    v = A_1 u \label{eq:orlati:3}\\
    A_2 u + C_2 \left(u^{q^I} + \alpha v^{q^J}\right) + \left(u^{q^J} + \beta v^{q^I} + \gamma v^{q^J}\right) = 0, \label{eq:orlati:4}
\end{numcases}
where
\begin{equation*}
\begin{split}
    A_1 &:= \frac{y \left(z^{q^I}+\alpha t^{q^J}\right) - t\left(x^{q^I}+\alpha y^{q^J}\right)}{x \left(z^{q^I}+\alpha t^{q^J}\right) - z\left(x^{q^I}+\alpha y^{q^J}\right)},\\
    A_2 &:= \frac{\left(x^{q^I}+\alpha y^{q^J}\right)\left(z^{q^J}+\beta t^{q^I} + \gamma t^{q^J}\right) - \left(z^{q^I}+\alpha t^{q^J}\right)\left(x^{q^J}+\beta y^{q^I} + \gamma y^{q^J}\right)}{x \left(z^{q^I}+\alpha t^{q^J}\right) - z\left(x^{q^I}+\alpha y^{q^J}\right)},\\
    C_2 &:= -\frac{x \left(z^{q^J}+\beta t^{q^I} + \gamma t^{q^J}\right) - z\left(x^{q^J}+\beta y^{q^I} + \gamma y^{q^J}\right)}{x \left(z^{q^I}+\alpha t^{q^J}\right) - z\left(x^{q^I}+\alpha y^{q^J}\right)}.
\end{split}
\end{equation*}
Therefore, in order to apply Bézout's Theorem, we show that the curve defined by \eqref{eq:orlati:3} is not a component of the curve defined by \eqref{eq:orlati:4}. Note that, if $A_2\neq 0$, this follows immediately. We consider hence the case $A_2=0$. The above system reads 
\begin{equation*}
    \begin{cases}
    v = A_1 u \\
    C_2 \left(u^{q^I} + \alpha A_1^{q^J}u^{q^J}\right) + \left(u^{q^J} + \beta A_1^{q^I}u^{q^I} + \gamma A_1^{q^J}u^{q^J}\right) = 0,
    \end{cases}
\end{equation*}
and the curve defined by \eqref{eq:orlati:3} is a component of the curve defined by \eqref{eq:orlati:4} if and only if the polynomial
\begin{equation*}
    C_2 \left(u^{q^I} + \alpha A_1^{q^J}u^{q^J}\right) + \left(u^{q^J} + \beta A_1^{q^I}u^{q^I} + \gamma A_1^{q^J}u^{q^J}\right)
\end{equation*}
is identically zero, i.e., if and only if $(A_1,C_2)$ satisfies the following system of equations:
\begin{equation}
\label{eq:system:1}
    \begin{cases}
    \beta A_1^{q^I} + C_2 = 0 \\
    1 + \gamma A_1^{q^J} + \alpha C_2 A_1^{q^J} = 0.
    \end{cases}
\end{equation}
From the first equation of \eqref{eq:system:1} we  have $C_2 = -\beta A_1^{q^I}$ and, substituting in the second equation, this gives
\begin{equation}
    \label{eq:polynomial:1}
    1 + \gamma A_1^{q^J} - \alpha\beta A_1^{q^J + q^I} = 0.
\end{equation}

Setting $X:= A_1^{-q^{I}}$ or $X:= A_1^{-q^{J}}$,  \eqref{eq:polynomial:1} corresponds to  $P^{I,J}_{\alpha,\beta,\gamma}(X)=0$ and by assumption it  has no solutions in $\F_{q^n}$. This  shows that the curve defined by \eqref{eq:orlati:3} is not a component of the curve defined by \eqref{eq:orlati:4}.  Therefore, by B\'ezout's theorem, the number of solutions of the system defined by \eqref{eq:orlati:3} and \eqref{eq:orlati:4} is at most $q^{\max\{I,J\}}$.

\item[(ii)] The case $y \left(z^{q^I}+\alpha t^{q^J}\right) - t\left(x^{q^I}+\alpha y^{q^J}\right)\neq 0$ can be treated analogously to the previous one, as we consider the $2\times 2$ submatrix of $\mathfrak{M}$ given by 
\begin{equation*}
    \begin{pmatrix}
    y &  x^{q^I}+\alpha y^{q^J} \\
    t &  z^{q^I}+\alpha t^{q^J}\\
    \end{pmatrix}.
\end{equation*}
\item[(iii)] If $x \left(z^{q^J}+\beta t^{q^I} + \gamma t^{q^J}\right) - z\left(x^{q^J}+\beta y^{q^I} + \gamma y^{q^J}\right)\neq 0$, we consider the $2\times 2$ submatrix of $\mathfrak{M}$ given by 
\begin{equation*}
    \begin{pmatrix}
    x &  x^{q^J}+\beta y^{q^I} + \gamma y^{q^J} \\
    z &  z^{q^J}+\beta t^{q^I} + \gamma t^{q^J} \\
    \end{pmatrix},
\end{equation*}
which has non-zero determinant. Proceeding as in the previous cases, we consider the system
\begin{equation*}
    \begin{cases}
     v = \tilde{A_1} u \\
    \tilde{A_2} u + \left(u^{q^I} + \alpha v^{q^J}\right) + \tilde{C_2}\left(u^{q^J} + \beta v^{q^I} + \gamma v^{q^J}\right) = 0,
    \end{cases}
\end{equation*}
where
\begin{equation*}
\begin{split}
    \tilde{A_1} &:= \frac{y \left(z^{q^J}+\beta t^{q^I} + \gamma t^{q^J}\right) - t\left(x^{q^J}+\beta y^{q^I} + \gamma y^{q^J}\right)}{x \left(z^{q^J}+\beta t^{q^I} + \gamma t^{q^J}\right) - z\left(x^{q^J}+\beta y^{q^I} + \gamma y^{q^J}\right)},\\
    \tilde{A_2} &:= -\frac{\left(x^{q^I}+\alpha y^{q^J}\right)\left(z^{q^J}+\beta t^{q^I} + \gamma t^{q^J}\right) - \left(z^{q^I}+\alpha t^{q^J}\right)\left(x^{q^J}+\beta y^{q^I} + \gamma y^{q^J}\right)}{x \left(z^{q^J}+\beta t^{q^I} + \gamma t^{q^J}\right) - z\left(x^{q^J}+\beta y^{q^I} + \gamma y^{q^J}\right)},\\
    \tilde{C_2} &:= -\frac{x \left(z^{q^I}+\alpha t^{q^J}\right) - z\left(x^{q^I}+\alpha y^{q^J}\right)}{x \left(z^{q^J}+\beta t^{q^I} + \gamma t^{q^J}\right) - z\left(x^{q^J}+\beta y^{q^I} + \gamma y^{q^J}\right)}.
\end{split}
\end{equation*}
Computations as in case (i) lead to the following system:
\begin{equation}
\label{eq:system:2}
    \begin{cases}
    1 + \beta \tilde{C_2} \tilde{A_1}^{q^I} = 0 \\
    \alpha \tilde{A_1}^{q^J} + \gamma \tilde{C_2} \tilde{A_1}^{q^J} + \tilde{C_2} = 0.
    \end{cases}
\end{equation}
From the first equation of \eqref{eq:system:2} we  have $\tilde{C_2} = -\frac{1}{\beta \tilde{A_1}^{q^I}}$ and, substituting in the second equation, this gives
\begin{equation*}
     -\alpha\beta \tilde{A_1}^{q^J + q^I} + 1 + \gamma \tilde{A_1}^{q^J} = 0.
\end{equation*}
Setting $X:= \tilde{A_1}^{-q^{I}}$ or $X:= \tilde{A_1}^{-q^{J}}$, the above equation is equivalent to   $P_{\alpha,\beta,\gamma}^{I,J}(X)=0$
and the conclusion follows as in case (i).

\item[(iv)] In the case $y \left(z^{q^J}+\beta t^{q^I} + \gamma t^{q^J}\right) - t\left(x^{q^J}+\beta y^{q^I} + \gamma y^{q^J}\right)\neq 0$, we proceed as above, this time starting from the $2\times 2$ submatrix of $\mathfrak{M}$ given by 
\begin{equation*}
    \begin{pmatrix}
    y &  x^{q^J}+\beta y^{q^I} + \gamma y^{q^J} \\
    t &  z^{q^J}+\beta t^{q^I} + \gamma t^{q^J} \\
    \end{pmatrix}.
\end{equation*}
\end{enumerate}

\end{proof}

The following is a direct consequence of Theorem \ref{thm:evasive}, combined with \cite[Theorem 3.3]{BMN2022}.
\begin{corollary}\label{cor:minimal}
 If  $P^{I,J}_{\alpha,\beta,\gamma}(X)$ has no roots in $\F_{q^n}$ and $\max\{I,J\}\leq (n-1)/2$, then $U_{\alpha, \beta, \gamma}^{I,J,n}$ is cutting, that is, for every $\Fn$-hyperplane of $V(4,q^n)$ we have $\langle H\cap U_{\alpha, \beta, \gamma}^{I,J,n} \rangle_{\Fn} = H$.  
\end{corollary}

We are now ready to prove our main result concerning the exceptionality and indecomposability of the family $U_{\alpha, \beta, \gamma}^{I,J,n}$.

\begin{theorem}\label{thm:exceptional_scattered}
For fixed $n, \alpha, \beta, \gamma, I\neq J$, with $\gcd(I,J,n)=1$, suppose that $P_{\alpha,\beta,\gamma}^{I,J}(X)$ has no roots in $\mathbb{F}_{q^n}$. Then the set $U_{\alpha, \beta, \gamma}^{I,J,n}$ is scattered and indecomposable over infinitely many extensions $\mathbb{F}_{q^{\ell n}}$ of $\mathbb{F}_{q^n}$.
\end{theorem}
\begin{proof}
The set $U_{\alpha, \beta, \gamma}^{I,J,n}$ is exceptional scattered by Theorem \ref{thm:scattered:1}. Note that for any $\ell$ large enough we have that $\max\{I,J\} \leq (n\ell-1)/2$ and thus $U_{\alpha, \beta, \gamma}^{I,J,n}$ is $(2,n-1)$-evasive by Theorem \ref{thm:evasive}. The claim follows by Lemma \ref{Lemma:ind}, with $m=s=2$ and $h=1$.
\end{proof}

We conclude by observing the main properties of the codes $C_{\alpha, \beta, \gamma}^{I,J,n}$ associated with the $\fq$-subspaces $U_{\alpha, \beta, \gamma}^{I,J,n}$.

\begin{remark} 
Let $I,J,\alpha,\beta,\gamma$ be such that  $P^{I,J}_{\alpha,\beta,\gamma}(X)$ has no roots in $\F_{q^n}$, and let us consider any code $C_{\alpha, \beta, \gamma}^{I,J,n}$ associated with $U_{\alpha, \beta, \gamma}^{I,J,n}$, that is, $C_{\alpha, \beta, \gamma}^{I,J,n}\in\Psi([U_{\alpha, \beta, \gamma}^{I,J,n}])$. First of all, $C_{\alpha, \beta, \gamma}^{I,J,n}$ is an MRD code of dimension $\dim_{\Fn}(C_{\alpha, \beta, \gamma}^{I,J,n})=4$. This is a consequence of \cite[Theorem 3.2]{csajbok2017maximum}. Furthermore, since $U_{\alpha, \beta, \gamma}^{I,J,n}$ is $1$-scattered (Theorem \ref{thm:scattered:1}), we also derive that the third generalized rank weight is $2n-1$, and by Theorem \ref{thm:evasive}, we also derive that the second generalized rank weight of $C_{\alpha, \beta, \gamma}^{I,J,n}$ is at least   
$2(n-\max\{I,J\})$; see \cite[Theorem 3]{randrianarisoa2020geometric}, \cite[Theorem 3.3]{marino2022evasive}. Thus, when $\max\{I,J\}\leq (n-1)/2$, we have that this second generalized rank weight is at least $n+1$.  By \cite[Proposition 4.10]{BMN2022} a decomposable code $\mathcal{D}=\mathcal{D}_1\oplus \mathcal{D}_2$ where $\mathcal{D}_1$  and $\mathcal{D}_2$ are $[n,2]_{q^n/q}$ MRD codes has second rank generalized weight equal to $n$. Since it is easy to see that generalized rank weights are invariant under code equivalence, we immediately derive that the codes $\mathcal{C}_{\alpha, \beta, \gamma}^{I,J,n}$ are new and inequivalent from already known codes.
Finally, by Corollary \ref{cor:minimal} and \cite[Corollary 5.7]{ABNR22}, the code $C_{\alpha, \beta, \gamma}^{I,J,n}$ is \emph{minimal}, that is, the set of supports of its nonzero codewords is an antichain and it has cardinality $\frac{q^{4n}-1}{q^n-1}$. We refer the reader to \cite{ABNR22} for a comprehensive understanding of minimal rank-metric codes.
\end{remark}

\subsection{Equivalence issue}\label{sec:equivalence}

Let $U_{\alpha, \beta, \gamma}^{I,J,n}$ and $U_{\overline{\alpha}, \overline{\beta}, \overline{\gamma}}^{I_0,J_0,n}$ be two sets as in Definition \ref{defn:sets},
\begin{equation}
\label{eq:sets:gammaell}
    \begin{split}
        &U_{\alpha, \beta, \gamma}^{I,J,n} =\left\{\left(x,y,x^{q^I}+\alpha y^{q^J},x^{q^J}+\beta y^{q^I} + \gamma y^{q^J}\right) \ : \ x,y\in \F_{q^n}\right\}\\
        &U_{\overline{\alpha}, \overline{\beta}, \overline{\gamma}}^{I_0,J_0,n}=\left\{\left(u,v,u^{q^{I_0}}+\overline{\alpha} v^{q^{J_0}},u^{q^{J_0}}+\overline{\beta} v^{q^{I_0}} + \overline{\gamma} v^{q^{J_0}}\right) \ : \ u,v\in \F_{q^n}\right\}.
    \end{split}
\end{equation}
The sets $U_{\alpha, \beta, \gamma}^{I,J,n}$ and $U_{\overline{\alpha}, \overline{\beta}, \overline{\gamma}}^{I_0,J_0,n}$ are $\Gamma$L$(4,q^n)$-equivalent if and only if there exist $\sigma\in \mathrm{Aut}(\F_{q^n})$ and
\begin{equation}
\label{eq:mat:gammaell}
\mathfrak{N}:=\begin{pmatrix}
a_{11}&a_{12}&a_{13}&a_{14}\\
a_{21}&a_{22}&a_{23}&a_{24}\\
a_{31}&a_{32}&a_{33}&a_{34}\\
a_{41}&a_{42}&a_{43}&a_{44}
\end{pmatrix}\in \mathrm{GL}(4,\F_{q^n})
\end{equation}
such that 
\begin{equation*}
\begin{pmatrix}
a_{11}&a_{12}&a_{13}&a_{14}\\
a_{21}&a_{22}&a_{23}&a_{24}\\
a_{31}&a_{32}&a_{33}&a_{34}\\
a_{41}&a_{42}&a_{43}&a_{44}
\end{pmatrix}
\begin{pmatrix}
x^\sigma\\
y^\sigma\\
\left(x^{q^I}+\alpha y^{q^J}\right)^\sigma\\
\left(x^{q^J}+\beta y^{q^I} + \gamma y^{q^J}\right)^\sigma
\end{pmatrix}=
\begin{pmatrix}
u\\
v\\
u^{q^{I_0}}+\overline{\alpha} v^{q^{J_0}}\\
u^{q^{J_0}}+\overline{\beta} v^{q^{I_0}} + \overline{\gamma} v^{q^{J_0}}
\end{pmatrix}.
\end{equation*}

Before proving Theorem \ref{thm:equivalence} on the $\Gamma\mathrm{L}(4,q^n)$-equivalence classes for the sets introduced in Definition \ref{defn:sets}, we establish some notations that will be useful in the proof.

Let $K:= J - I$ and
\begin{eqnarray}
  \small  \rho^{q^K}:= \left(\frac{\beta}{\overline{\beta}}\right)^{q^{K-I}}\left(\frac{\overline{\alpha}}{\alpha}\right)^{q^{-I}}, \quad \vartheta^{q^K}:= \left(\frac{\beta}{\overline{\beta}}\right)^{q^{K-I}}\left(\frac{\overline{\alpha}}{\alpha}\right)^{q^{-I}}\overline{\gamma}^{q^{K-I}}, \quad \sigma^{q^K}:= - \left(\frac{\overline{\alpha}}{\alpha}\right)^{q^{-I}}\gamma^{q^{-I}},\nonumber\\
    \mu^{q^K}:= \left(\frac{\overline{\alpha}\overline{\beta}}{\overline{\gamma}}\right)^{q^{-I}} , \qquad \nu^{q^K}:= \left(\frac{\gamma \overline{\alpha}}{\overline{\gamma}\alpha}\right)^{q^{-I}} , \qquad \xi^{q^K}:= - \left(\frac{\beta^{q^K}\overline{\alpha}^{q^K+1}}{\overline{\gamma}\alpha}\right)^{q^{-I}}.\label{eq:not:1}
\end{eqnarray}

\begin{theorem}
\label{thm:equivalence}
Let $0< I,J,I_0,J_0\leq (n-1)/2$, $I\neq J$ and $I_0\neq J_0$. 
Consider two sets $U_{\alpha, \beta, \gamma}^{I,J,n}$ and $U_{\overline{\alpha}, \overline{\beta}, \overline{\gamma}}^{I_0,J_0,n}$, with notations as in \eqref{eq:sets:gammaell}, \eqref{eq:mat:gammaell}, and \eqref{eq:not:1}.

Then $U_{\alpha, \beta, \gamma}^{I,J,n}$ and $U_{\overline{\alpha}, \overline{\beta}, \overline{\gamma}}^{I_0,J_0,n}$ are not $\Gamma\mathrm{L}(4,q^n)$-equivalent if one of the following conditions holds:
\begin{enumerate}
    \item $(I,J)\neq (I_0, J_0)$;
    \item $(I,J)= (I_0, J_0)$ and 
    $$
    \begin{cases}
    X = \rho^{q^K}X^{q^{2K}} + \vartheta^{q^K}Y^{q^{2K}} + \sigma^{q^K}Y^{q^K}\\
    X = \mu^{q^K}Y + \nu^{q^K}X^{q^K} + \xi^{q^K}Y^{q^{2K}}.
    \end{cases}
    $$
    has no solutions $(x,y)\in \mathbb{F}_{q^n}^2\setminus\{(0,0)\}$
\end{enumerate}
\end{theorem}

\begin{proof}
As noted above, two sets $U_{\alpha, \beta, \gamma}^{I,J,n}$ and $U_{\overline{\alpha}, \overline{\beta}, \overline{\gamma}}^{I_0,J_0,n}$ (as in \eqref{eq:sets:gammaell}) are $\Gamma$L$(4,q^n)$-equivalent if and only if there exist $\sigma\in \mathrm{Aut}(\F_{q^n})$ and $\mathfrak{N} \in \mathrm{GL}(4,\F_{q^n})$, with notations as in \eqref{eq:mat:gammaell},
such that 
\begin{equation}
\label{eq:gammaell:1}
\begin{pmatrix}
a_{11}&a_{12}&a_{13}&a_{14}\\
a_{21}&a_{22}&a_{23}&a_{24}\\
a_{31}&a_{32}&a_{33}&a_{34}\\
a_{41}&a_{42}&a_{43}&a_{44}
\end{pmatrix}
\begin{pmatrix}
x^\sigma\\
y^\sigma\\
\left(x^{q^I}+\alpha y^{q^J}\right)^\sigma\\
\left(x^{q^J}+\beta y^{q^I} + \gamma y^{q^J}\right)^\sigma
\end{pmatrix}=
\begin{pmatrix}
u\\
v\\
u^{q^{I_0}}+\overline{\alpha} v^{q^{J_0}}\\
u^{q^{J_0}}+\overline{\beta} v^{q^{I_0}} + \overline{\gamma} v^{q^{J_0}}
\end{pmatrix}.
\end{equation}

Throughout the proof, we denote $\tilde{x}:=x^\sigma$ and $\tilde{y}:=y^\sigma$ and we consider $\sigma$ as acting on $\alpha,\beta$ and $\gamma$ as the identity. This can be done without loss of generality, since we are working on two generic sets and $\sigma$ is an automorphism of $\F_{q^n}$ (hence, we can just rename $\alpha^{\sigma}, \beta^{\sigma}, \gamma^{\sigma}$ as $\alpha, \beta, \gamma$).

From \eqref{eq:gammaell:1}, we have the following system of equations:
\begin{equation*}
    \begin{cases}
a_{11}\tilde{x} + a_{12}\tilde{y} + a_{13}\left(\tilde{x}^{q^I}+\alpha\tilde{y}^{q^J}\right) + a_{14}\left(\tilde{x}^{q^J}+\beta\tilde{y}^{q^I}+\gamma\tilde{y}^{q^J}\right)=u\\
a_{21}\tilde{x} + a_{22}\tilde{y} + a_{23}\left(\tilde{x}^{q^I}+\alpha\tilde{y}^{q^J}\right) + a_{24}\left(\tilde{x}^{q^J}+\beta\tilde{y}^{q^I}+\gamma\tilde{y}^{q^J}\right)=v\\
a_{31}\tilde{x} + a_{32}\tilde{y} + a_{33}\left(\tilde{x}^{q^I}+\alpha\tilde{y}^{q^J}\right) + a_{34}\left(\tilde{x}^{q^J}+\beta\tilde{y}^{q^I}+\gamma\tilde{y}^{q^J}\right)=u^{q^{I_0}}+\overline{\alpha} v^{q^{J_0}}\\
a_{41}\tilde{x} + a_{42}\tilde{y} + a_{43}\left(\tilde{x}^{q^I}+\alpha\tilde{y}^{q^J}\right) + a_{44}\left(\tilde{x}^{q^J}+\beta\tilde{y}^{q^I}+\gamma\tilde{y}^{q^J}\right)=u^{q^{J_0}}+\overline{\beta} v^{q^{I_0}} + \overline{\gamma} v^{q^{J_0}}.
    \end{cases}
\end{equation*}
Substituting, we obtain the following two equations:
\begin{equation}
\label{eq:equiv:1}
    \small \begin{split}
        a_{31}\tilde{x} + a_{32}\tilde{y} + a_{33}\tilde{x}^{q^I} + a_{33}\alpha\tilde{y}^{q^J} + a_{34}\tilde{x}^{q^J} + a_{34}\beta\tilde{y}^{q^I} + a_{34}\gamma \tilde{y}^{q^J} - a_{11}^{q^{I_0}}\tilde{x}^{q^{I_0}} - a_{12}^{q^{I_0}}\tilde{y}^{q^{I_0}} &+ \\
        - a_{13}^{q^{I_0}}\left(\tilde{x}^{q^{I + I_0}} + \alpha^{q^{I_0}}\tilde{y}^{q^{J + I_0}}\right) - a_{14}^{q^{I_0}}\left(\tilde{x}^{q^{J + I_0}} + \beta^{q^{I_0}}\tilde{y}^{q^{I + I_0}} + \gamma^{q^{I_0}}\tilde{y}^{q^{J + I_0}}\right) &+ \\
        - \overline{\alpha}\left(a_{21}^{q^{J_0}}\tilde{x}^{q^{J_0}} + a_{22}^{q^{J_0}}\tilde{y}^{q^{J_0}} + a_{23}^{q^{J_0}}\left(\tilde{x}^{q^{I + J_0}} + \alpha^{q^{J_0}}\tilde{y}^{q^{J + J_0}}\right) + a_{24}^{q^{J_0}}\left(\tilde{x}^{q^{J + J_0}} + \beta^{q^{J_0}}\tilde{y}^{q^{I + J_0}} + \gamma^{q^{J_0}}\tilde{y}^{q^{J + J_0}}\right)\right)&=0,
    \end{split}
\end{equation}
\begin{equation}
\label{eq:equiv:2}
    \small \begin{split}
        a_{41}\tilde{x} + a_{42}\tilde{y} + a_{43}\tilde{x}^{q^I} + a_{43}\alpha\tilde{y}^{q^J} + a_{44}\tilde{x}^{q^J} + a_{44}\beta\tilde{y}^{q^I} + a_{44}\gamma \tilde{y}^{q^J} - a_{11}^{q^{J_0}}\tilde{x}^{q^{J_0}} - a_{12}^{q^{J_0}}\tilde{y}^{q^{J_0}} &+ \\
        - a_{13}^{q^{J_0}}\left(\tilde{x}^{q^{I + J_0}} + \alpha^{q^{J_0}}\tilde{y}^{q^{J + J_0}}\right) - a_{14}^{q^{J_0}}\left(\tilde{x}^{q^{J + J_0}} + \beta^{q^{J_0}}\tilde{y}^{q^{I + J_0}} + \gamma^{q^{J_0}}\tilde{y}^{q^{J + J_0}}\right) &+ \\
        - \overline{\beta}\left(a_{21}^{q^{I_0}}\tilde{x}^{q^{I_0}} + a_{22}^{q^{I_0}}\tilde{y}^{q^{I_0}} + a_{23}^{q^{I_0}}\left(\tilde{x}^{q^{I + I_0}} + \alpha^{q^{I_0}}\tilde{y}^{q^{J + I_0}}\right) + a_{24}^{q^{I_0}}\left(\tilde{x}^{q^{J + I_0}} + \beta^{q^{I_0}}\tilde{y}^{q^{I + I_0}} + \gamma^{q^{I_0}}\tilde{y}^{q^{J + I_0}}\right)\right)&+\\
        - \overline{\gamma}\left(a_{21}^{q^{J_0}}\tilde{x}^{q^{J_0}} + a_{22}^{q^{J_0}}\tilde{y}^{q^{J_0}} + a_{23}^{q^{J_0}}\left(\tilde{x}^{q^{I + J_0}} + \alpha^{q^{J_0}}\tilde{y}^{q^{J + J_0}}\right) + a_{24}^{q^{J_0}}\left(\tilde{x}^{q^{J + J_0}} + \beta^{q^{J_0}}\tilde{y}^{q^{I + J_0}} + \gamma^{q^{J_0}}\tilde{y}^{q^{J + J_0}}\right)\right)&=0. 
    \end{split}
\end{equation}

We wish to show that, in both cases (a) and (b) listed in the statement of the theorem, it is not possible to find an element of $\mathrm{GL}(4,\F_{q^n})$ such that \eqref{eq:equiv:1} and \eqref{eq:equiv:2} are both satisfied for any values of $\tilde{x},\tilde{y}\in \Fn$, i.e., such that the polynomials on the left hand side of  \eqref{eq:equiv:1} and \eqref{eq:equiv:2} are both identically zero. Note that this last equivalence holds because the left hand sides of  \eqref{eq:equiv:1} and \eqref{eq:equiv:2} are polynomials in $\tilde{x}$ and $\tilde{y}$ of degree smaller than $q^n$.

We prove this by considering separately the listed conditions (a) and (b).

\begin{enumerate}
    \item  $(I,J)\neq (I_0, J_0)$.

\begin{itemize}
    \item \textbf{Case} $I \neq I_0, J_0$.\\
    Considering the terms of  \eqref{eq:equiv:1}, note that we have:
    \begin{equation*}
    a_{31} \tilde{x}=0 \quad
    a_{32} \tilde{y}=0\quad
    a_{33} \tilde{x}^{q^I}=0\quad
a_{34}\beta \tilde{y}^{q^{I}}=0.
\end{equation*}
Hence, $a_{31} = a_{32} = a_{33} = a_{34} = 0$.

    \item \textbf{Case} $I = J_0$.\\

\begin{itemize}
    \item \textbf{Subcase} $J\neq I + I_0$ and $J\neq 2I$.\\
    Considering the coefficients of $\tilde{x}$, $\tilde{y}$, $\tilde{x}^{q^J}$, $\tilde{y}^{q^J}$, we get $a_{31} = a_{32} = a_{33} = a_{34} = 0$.

\item \textbf{Subcase} $J = I + I_0$ and $J + I_0 = J_0 + I$.\\
Considering the coefficients of $\tilde{x}$, $\tilde{y}$, $\tilde{x}^{q^J}$, $\tilde{y}^{q^J}$, $\tilde{x}^{q^{I+J_0}}$, $\tilde{y}^{q^{I+J_0}}$, $\tilde{x}^{q^{J+J_0}}$, $\tilde{y}^{q^{J+J_0}}$, we get
$a_{31} = a_{32} = a_{33} = a_{34} = 0$.

\item \textbf{Subcase} $J = I + I_0$ and $J + I_0 \neq J_0 + I$.\\
Considering the coefficients of $\tilde{x}$, $\tilde{y}$, $\tilde{x}^{q^J}$, $\tilde{y}^{q^J}$,  $\tilde{x}^{q^{J+I_0}}$, $\tilde{y}^{q^{J+I_0}}$, we get
 $a_{31} = a_{32} = a_{33} = a_{34} = 0$.

\item \textbf{Subcase} $J = 2I$ and $J\neq I+I_0$.\\
Considering the coefficients of $\tilde{x}$, $\tilde{y}$, $\tilde{x}^{q^J}$, $\tilde{y}^{q^J}$,  $\tilde{x}^{q^{J+J_0}}$, $\tilde{y}^{q^{J+J_0}}$, we get
$a_{31} = a_{32} = a_{33} = a_{34} = 0$.
\end{itemize}

\item \textbf{Case} $I = I_0$ and $J \neq J_0$.\\
\begin{itemize}
    \item \textbf{Subcase} $J\neq I + J_0$ and $J\neq 2I$.\\
Considering the coefficients of $\tilde{x}$, $\tilde{y}$, $\tilde{x}^{q^J}$, $\tilde{y}^{q^J}$, we get
$a_{31} = a_{32} = a_{33} = a_{34} = 0$.

\item \textbf{Subcase} $J = I + J_0$.\\
Considering the coefficients of $\tilde{x}$, $\tilde{y}$, $\tilde{x}^{q^J}$, $\tilde{y}^{q^J}$,  $\tilde{x}^{q^{J+J_0}}$, $\tilde{y}^{q^{J+J_0}}$, we get
$a_{31} = a_{32} = a_{33} = a_{34} = 0$. 

\item \textbf{Subcase} $J = 2I$ and $J_0\neq 3I$.\\
Considering the coefficients of $\tilde{x}$, $\tilde{y}$, $\tilde{x}^{q^J}$, $\tilde{y}^{q^J}$,  $\tilde{x}^{q^{I+J}}$, $\tilde{y}^{q^{I+J}}$, we get
$a_{31} = a_{32} = a_{33} = a_{34} = 0$.

\item \textbf{Subcase} $J = 2I$ and $J_0= 3I$.\\
Considering the coefficients of $\tilde{x}$, $\tilde{y}$, $\tilde{x}^{q^{5I}}$, $\tilde{y}^{q^{5I}}$,  $\tilde{x}^{q^{2I}}$, $\tilde{y}^{q^{2I}}$, we get
$a_{31} = a_{32} = a_{33} = a_{34} = 0$.

\end{itemize}

\end{itemize}

In all the cases listed above,  the matrix \eqref{eq:mat:gammaell} 
is not an element of $\mathrm{GL}(4,\F_{q^n})$. Hence the two sets $U_{\alpha, \beta, \gamma}^{I,J,n}$ and $U_{\overline{\alpha}, \overline{\beta}, \overline{\gamma}}^{I_0,J_0,n}$ are not $\Gamma\mathrm{L}(4,q^n)$-equivalent, for any choice of elements $\alpha,\beta,\gamma,\overline{\alpha},\overline{\beta},\overline{\gamma}\in \F_{q^n}^*$.
\item  $(I,J)= (I_0, J_0)$ and 
    $$
    \begin{cases}
    X = \rho^{q^K}X^{q^{2K}} + \vartheta^{q^K}Y^{q^{2K}} + \sigma^{q^K}Y^{q^K}\\
    X = \mu^{q^K}Y + \nu^{q^K}X^{q^K} + \xi^{q^K}Y^{q^{2K}}
    \end{cases}
    $$
    has no solutions $(x,y)\in \mathbb{F}_{q^n}^2\setminus\{(0,0)\}$.

From \eqref{eq:equiv:1},  for the coefficients of $$\tilde{x}, \tilde{y}, \tilde{x}^{q^I}, \tilde{y}^{q^I}, \tilde{x}^{q^J}, \tilde{y}^{q^J}, \tilde{x}^{q^{2I}}, \tilde{y}^{q^{2I}}, \tilde{x}^{q^{I+J}}, \tilde{y}^{q^{I+J}}, \tilde{x}^{q^{2J}}, \tilde{y}^{q^{2J}}$$

we obtain the following conditions
\begin{eqnarray}
\label{system:1}
        &&a_{31} = a_{32} = a_{24} = a_{23} = a_{13} = a_{14} = 0\nonumber \\
        &&a_{33} = a_{11}^{q^I}, \quad 
        a_{33}\alpha + a_{34}\gamma - \overline{\alpha}a_{22}^{q^J} = 0, \quad 
        a_{34} = \overline{\alpha}a_{21}^{q^J}, \quad
        a_{34}\beta = a_{12}^{q^I}.
\end{eqnarray}
Then, from \eqref{eq:equiv:2}, we obtain the following system:
\begin{equation}
\label{system:2}
    \begin{cases}
        a_{41} = a_{42} = 0\\
        a_{43} = \overline{\beta}a_{21}^{q^I}\\
        a_{43}\alpha + a_{44}\gamma - a_{12}^{q^J} - \overline{\gamma}a_{22}^{q^J} = 0\\
        a_{44} - a_{11}^{q^J} - \overline{\gamma}a_{21}^{q^J} = 0\\
        a_{44}\beta = \overline{\beta}a_{22}^{q^I}.
    \end{cases}
\end{equation}
Hence, considering the conditions on the coefficients {given by \eqref{system:1} and \eqref{system:2},} 
we have the following:
    \begin{numcases}{}
    a_{11}^{q^I}\alpha + \overline{\alpha}a_{21}^{q^J}\gamma = \overline{\alpha}a_{22}^{q^J} \label{eq:1:gammaell}\\
    \overline{\beta}a_{21}^{q^I}\alpha + \left(a_{11}^{q^J} + \overline{\gamma}a_{21}^{q^J}\right)\gamma = a_{12}^{q^J} + \overline{\gamma}a_{22}^{q^J} \label{eq:2:gammaell}\\
    \left(a_{11}^{q^J} + \overline{\gamma}a_{21}^{q^J}\right)\beta = \overline{\beta}a_{22}^{q^I}. \label{eq:3:gammaell}
    \end{numcases}
As $a_{12}^{q^I} = a_{34}\beta = \beta\overline{\alpha}a_{21}^{q^J}$, Equation \eqref{eq:2:gammaell} can be rewritten as 
\begin{equation}
\label{eq:obtained:1}
    \overline{\beta}a_{21}^{q^I}\alpha + \left(a_{11}^{q^J} + \overline{\gamma}a_{21}^{q^J}\right)\gamma = \beta^{q^K}\overline{\alpha}^{q^K}a_{21}^{q^{J+K}} + \overline{\gamma}a_{22}^{q^J}.
\end{equation}
Moreover, from Equation \eqref{eq:3:gammaell}, we have that
\begin{equation*}
  a_{22}^{q^J} = \frac{\beta^{q^K}}{\overline{\beta}^{q^K}}\left(a_{11}^{q^{J+K}} + \overline{\gamma}^{q^K}a_{21}^{q^{J+K}}\right).
\end{equation*}
Then, from Equation \eqref{eq:1:gammaell}, we obtain
\begin{equation}
\label{eq:obtained:a_{22}}
\frac{\beta^{q^K}}{\overline{\beta}^{q^K}}\left(a_{11}^{q^{J+K}} + \overline{\gamma}^{q^K}a_{21}^{q^{J+K}}\right) = \frac{a_{11}^{q^I}\alpha + \overline{\alpha}a_{21}^{q^J}\gamma}{\overline{\alpha}}
\end{equation}
and substituting in Equation \eqref{eq:obtained:1} we have
\begin{equation}
\label{eq:obtained:2}
    \overline{\beta}a_{21}^{q^I}\alpha + \left(a_{11}^{q^J} + \overline{\gamma}a_{21}^{q^J}\right)\gamma = \beta^{q^K}\overline{\alpha}^{q^K}a_{21}^{q^{J+K}} + \frac{\overline{\gamma}}{\overline{\alpha}}\left(a_{11}^{q^I}\alpha + \overline{\alpha}a_{21}^{q^J}\gamma\right).
\end{equation}

Considering now Equation \eqref{eq:obtained:a_{22}}, we rewrite it as
\begin{equation*}
    \begin{split}
        %
        a_{11}^{q^I} &= \left(\frac{\beta}{\overline{\beta}}\right)^{q^K}\left(\frac{\overline{\alpha}}{\alpha}\right)\left(a_{11}^{q^{J+K}} + \overline{\gamma}^{q^K}a_{21}^{q^{J+K}}\right) - \frac{\overline{\alpha}}{\alpha}\gamma a_{21}^{q^J}.
    \end{split}
\end{equation*}
From this last equation, we have
\begin{equation}
\label{eq:frobenius:1}
           a_{11} = \left(\frac{\beta}{\overline{\beta}}\right)^{q^{K-I}}\left(\frac{\overline{\alpha}}{\alpha}\right)^{q^{-I}}\left(a_{11}^{q^{2K}} + \overline{\gamma}^{q^{K-I}}a_{21}^{q^{2K}}\right) - \left(\frac{\overline{\alpha}}{\alpha}\right)^{q^{-I}}\gamma^{q^{-I}} a_{21}^{q^K}.
\end{equation}
From Equation \eqref{eq:obtained:2}, we obtain instead
\begin{equation*}
  \begin{split}
 a_{11}^{q^I} &= \frac{\overline{\alpha}\overline{\beta}}{\overline{\gamma}}a_{21}^{q^I} + \frac{\gamma \overline{\alpha}}{\overline{\gamma}\alpha}a_{11}^{q^J} - \frac{\beta^{q^K}\overline{\alpha}^{q^K+1}}{\overline{\gamma}\alpha}a_{21}^{q^{J+K}}.
    \end{split}
\end{equation*}
From this last equation, we then have
\begin{equation}
\label{eq:frobenius:2}
           a_{11} = \left(\frac{\overline{\alpha}\overline{\beta}}{\overline{\gamma}}\right)^{q^{-I}}a_{21} + \left(\frac{\gamma \overline{\alpha}}{\overline{\gamma}\alpha}\right)^{q^{-I}}a_{11}^{q^K} - \left(\frac{\beta^{q^K}\overline{\alpha}^{q^K+1}}{\overline{\gamma}\alpha}\right)^{q^{-I}}a_{21}^{q^{2K}}.
\end{equation}

Now, with the notations introduced in \eqref{eq:not:1}, we can rewrite Equations \eqref{eq:frobenius:1} and  \eqref{eq:frobenius:2} as
\begin{equation}\label{O'Sistemone}
    \begin{cases}
    a_{11} = \rho^{q^K}a_{11}^{q^{2K}} + \vartheta^{q^K}a_{21}^{q^{2K}} + \sigma^{q^K}a_{21}^{q^K}\\
    a_{11} = \mu^{q^K}a_{21} + \nu^{q^K}a_{11}^{q^K} + \xi^{q^K}a_{21}^{q^{2K}}. 
    \end{cases}
\end{equation}
If the system above has the unique solution $(0,0)$ in $\mathbb{F}_{q^n}^2$, we also get $a_{22}=0=a_{12}=a_{44}=a_{34}=a_{33}$, a contradiction to $\mathfrak{N}\in \mathrm{GL}(4,\mathbb{F}_{q^n})$.
\end{enumerate}
\end{proof}

To determine whether System \eqref{O'Sistemone} has a non-trivial solution in $\mathbb{F}_{q^n}^2$ is not an easy task. In the following, we only provide an example which shows that non-trivial solutions of \eqref{O'Sistemone} could  yield the equivalence between  two sets $U_{\alpha, \beta, \gamma}^{I,J,n}$ and $U_{\overline{\alpha}, \overline{\beta}, \overline{\gamma}}^{I,J,n}$.

\begin{corollary}
\label{cor:equivalence:3}
Let $U_{\alpha, \beta, \gamma}^{I,J,n}$ and $U_{\overline{\alpha}, \overline{\beta}, \overline{\gamma}}^{I,J,n}$ be two scattered sets as above, with notations as in Theorem \ref{thm:equivalence}. 
If $\rho = \nu^{q^K+1}$ and $\nu$ is a $\left(q^K-1\right)$-th power in $\F_{q^n}$, then $U_{\alpha, \beta, \gamma}^{I,J,n}$ and $U_{\overline{\alpha}, \overline{\beta}, \overline{\gamma}}^{I,J,n}$ are $\Gamma\mathrm{L}(4,q^n)$-equivalent.
\end{corollary}

\begin{proof}
Since $\rho = \nu^{q^K+1}$ and $\nu$ is a $\left(q^K-1\right)$-th power in $\mathbb{F}_{q^n}$, $(a_{11},a_{21})=\Big(\sqrt[q^{K}-1]{1/\nu^{q^k}},0\Big)$ is a solution of  System \eqref{O'Sistemone}. 
From \eqref{system:1} and \eqref{system:2} 
\begin{eqnarray*}a_{12}=a_{13}=a_{14}=a_{21}=a_{23}=a_{24}=a_{31} = a_{32} = a_{34} = a_{41}=a_{42}=a_{43}=0\\
        a_{33} = a_{11}^{q^I}, \qquad 
        a_{33}\alpha = \overline{\alpha}a_{22}^{q^J}, \qquad 
             a_{44}\gamma = \overline{\gamma}a_{22}^{q^J}, \qquad 
        a_{44} = a_{11}^{q^J}, \qquad 
        a_{44}\beta = \overline{\beta}a_{22}^{q^I},
\end{eqnarray*}
that is, $a_{33} = a_{11}^{q^I}$, $a_{44} = a_{11}^{q^J}$, $a_{22}=\left(\frac{\gamma}{\overline{\gamma}}\right)^{q^{-J}}a_{11}$. The last two conditions read 

$$a_{11}^{q^K-1}=\frac{1}{\nu^{q^K}} \textrm{ and } a_{11}^{q^K-1}=\left(\frac{\overline{\beta}}{\beta}\right)^{q^{-I}}\left(\frac{\gamma}{\overline{\gamma}}\right)^{q^{-K-I}}$$
and they are compatible by our assumptions on $\rho$ and $\nu$.
\end{proof}

\subsection{The ``ordinary" duality}

The map $\Tr_{{q^n}/q}(X_0X_3-X_1X_2)$ defines a quadratic form of
$\Fn^4$ (regarded as $\fq$-vector space) over $\F_q$. The polar form
associated with such a quadratic form is
$\Tr_{{q^n}/q}(\sigma({\bf X},{\bf Y}))$, where $$\sigma ({\bf X},{\bf Y})=\left((X_0,X_1,X_2,X_3),(Y_0,Y_1,Y_2,Y_3)\right)=X_0Y_3+X_3Y_0-X_1Y_2-X_2Y_1.$$ 

If $f\in\cL_{n,q}[X]$ we will denote by $f^\top$ the  \textbf{adjoint} of $f$ with
respect to the $\fq$-bilinear form $\Tr_{q^n/q}(xy)$ on $\Fn$, that is defined by

$$\Tr_{q^n/q}(xf(y))=\Tr_{q^n/q}(yf^\top(x)) \quad \mbox{ for any } x,y \in \Fn.$$

\begin{comment}
Note that if $\F_{q'}$ is a subfield of $\F_q$ then the adjoint of
$f$ (as $\F_{q'}$--linear map) with respect to the
$\F_{q'}$--bilinear form $\Tr_{q^n/q'}$ is $f^t$ as well; indeed by
\cite[Thm. 2.26]{LN}, we have
$$\Tr_{q^n/q'}(xf(y))=\Tr_{q/q'}(\Tr_{q^n/q}(xf(y)))=\Tr_{q/q'}(\Tr_{q^n/q}(yf^t(x)))=\Tr_{q^n/q'}(yf^t(x)).$$
\end{comment}

Let $h_1,h_2,g_1,g_2\in \cL_{n,q}[X]$, and let 
$$X=\{\left(x,y,h_1(x)+h_2(y),g_1(x)+g_2(y)\right) : x,y\in\F_{q^n}\}$$
be a $2n$-dimensional $\fq$-subspace of $\Fn^4$. Straightforward computations show that the orthogonal complement of $X$ with respect to the
$\F_q$-bilinear form $\Tr_{q^n/q}(\sigma({\bf X},{\bf Y}))$ is \[X^\perp=\left\{\left(x,y,g_2^\top(x)-h_2^\top(y),h_1^\top(y)-g_1^\top(x)\right) \ : \ x,y\in\F_{q^n}\right\}.\]
Hence, the orthogonal complement of
$$        U_{\alpha, \beta, \gamma}^{I,J,n}:=\left\{\left(x,y,x^{q^I}+\alpha y^{q^J},x^{q^J}+\beta y^{q^I} + \gamma y^{q^J}\right)  : x,y\in \F_{q^n}\right\} $$
is
$$
        (U_{\alpha, \beta, \gamma}^{I,J,n}) ^{\perp}:=\left\{\left(x,y,\beta^{q^{n-I}}x^{q^{n-I}}+\gamma ^{q^{n-J}}x^{q^{n-J}}-\alpha ^{q^{n-J}}y^{q^{n-J}},y^{q^{n-I}}-x^{q^{n-J}}\right) :  x,y\in \F_{q^n}\right\},
$$
which is equivalent to 
$$        
\left\{\left(x,y,x^{q^{n-I}}-y^{q^{n-J}},x^{q^{n-J}}- \frac{\beta^{q^{n-I}}}{\alpha ^{q^{n-J}}}y^{q^{n-I}}-\frac{\gamma ^{q^{n-J}}}{\alpha ^{q^{n-J}}}y^{q^{n-J}}\right) :  x,y\in \F_{q^n}\right\}=U_{\overline{\alpha}, \overline{\beta}, \overline{\gamma}}^{I_0,J_0,n}, $$
where 
$$ I_0:=n-I, \quad J_0:=n-J, \quad \overline{\alpha}:=-1, \quad \overline{\beta}:=-\frac{\beta^{q^{n-I}}}{\alpha ^{q^{n-J}}}, \quad \overline{\gamma}:=-\frac{\gamma ^{q^{n-J}}}{\alpha ^{q^{n-J}}}.$$

\section{Open problems}\label{sec:conclusions}

In this paper, we have provided an infinite family $U_{\alpha, \beta, \gamma}^{I,J,n}$, as in Definition \ref{defn:sets}, of $2n$-dimensional (indecomposable) exceptional scattered subspaces in $V(4,q^n)$; see Theorem \ref{thm:scattered:1} and Theorem \ref{thm:exceptional_scattered}. We have also derived a condition on their evasivity with respect to $2$-dimensional $\Fn$-subspces in Theorem \ref{thm:evasive}, depending on $\max\{I,J\}$. All these results need the additional hypothesis on the polynomial $P_{\alpha, \beta, \gamma}^{I,J,n}(X)$ given in \eqref{Eq:P} having no roots in $\Fn$. We have observed in Remark \ref{rem:no_roots} that we can easily find some conditions to ensure this. However, this is far from characterizing such polynomials and finding the exact number of $\alpha,\beta,\gamma$ such that $P_{\alpha, \beta, \gamma}^{I,J,n}(X)$ has no roots in $\Fn$. 

\begin{problem} For any pair $1\leq I,J<n $, find explicit necessary conditions on  $\alpha,\beta,\gamma\in\Fn^*$ such that the polynomial $P_{\alpha, \beta, \gamma}^{I,J,n}(X)$ has no roots in $\Fn$. Furthermore, determine the exact number of such triples.
\end{problem}

Necessary and sufficient conditions for this to hold were given in \cite[Theorem 8]{mcguire2019characterization} and \cite[Theorem 9]{kim2021solving}, but these are not explicit, and they do not seem to help in the counting.

We have also showed, in Theorem \ref{thm:scattered:2}, that the condition on the polynomial $P_{\alpha, \beta, \gamma}^{I,J,n}(X)$ not having roots in $\Fn$ is necessary, when we are in the small $q$-degree regime, that is, when $0\leq I,J\leq n/4$. The techniques used are not suitable for showing that this result holds also for larger values of $I,J$. However, we have no concrete counterexamples indicating that this is not true.

\begin{problem} Extend the result of Theorem \ref{thm:scattered:2} to larger $q$-degree regimes, that is when $0\leq I,J\leq n-1$. 
\end{problem}

Finally, in Section \ref{sec:equivalence}, we analyzed the $\Gamma\mathrm{L}(4,q^n)$-equivalence of the $\fq$-subspaces $U_{\alpha, \beta, \gamma}^{I,J,n}$. In this paper we found some sufficient conditions for equivalence (Theorem \ref{thm:equivalence}) and inequivalence (Corollary \ref{cor:equivalence:3}). However, the picture is far to be complete.

\begin{problem}
 Complete the study of $\Gamma\mathrm{L}(4,q^n)$-equivalence of the  $\fq$-subspaces $U_{\alpha, \beta, \gamma}^{I,J,n}$.
\end{problem}

\section*{Acknowledgments}
This research was supported by the Italian National Group for Algebraic and Geometric Structures and their Applications (GNSAGA - INdAM).

\bibliographystyle{abbrv}
\bibliography{biblio.bib}

\end{document}